\documentclass[12pt]{article}
\usepackage{graphicx} 
\usepackage{amsmath}
\usepackage{amsthm}
\usepackage{amssymb}
\usepackage{geometry}
\usepackage{mathrsfs}
\usepackage{lmodern}
\usepackage{upgreek}
\usepackage[
backend=bibtex8,
sorting=nyt,
]{biblatex}
\usepackage[colorlinks=true,
linkcolor=Peach,
citecolor=Peach]{hyperref}
\usepackage[dvipsnames, svgnames]{xcolor}
\hypersetup{
colorlinks=true,
linktoc=all,
linkcolor=RoyalBlue
}
\usepackage{csquotes}
\usepackage[T1]{fontenc} 
\usepackage{dsfont}
\newcommand{\N}{\mathbb N}
\newcommand{\R}{\mathbb R}
\newcommand{\E}{\mathbb E}
\newcommand{\p}{\mathbb P}
\newcommand{\1}{\mathds{1}}

\newcommand{\Z}{\mathbb Z}

\renewcommand{\d}{\mathrm d}
\renewcommand{\i}{\mathrm i}
\newcommand{\indep}{\perp \!\!\! \perp}
\newcommand{\distribution}{\underset{n\rightarrow\infty}{\implies}}
\newcommand{\Var}{\mathrm{Var}}

\newtheorem{theorem}{Theorem}
\newtheorem*{theorem*}{Theorem}

\newtheorem{lemma}{Lemma}
\newtheorem*{remark}{Remark}
\newtheorem{corollary}{Corollary}

\numberwithin{equation}{section}

\title{Functional Limit Theorems for the range of stable random walks}
\author{Maxence Baccara$^*$}

\begin{document}

\maketitle
\let\thefootnote\relax\footnotetext{$^{*}$ CMAP, CPHT, MERGE, École polytechnique, Institut Polytechnique de Paris, 91120 Palaiseau, France}

\begin{abstract}
    In this paper we establish Functional Limit Theorems for the range of random walks in $\Z^d$ that are in the domain of attraction of a non-degenerate $\beta$-stable process in the weakly transient and recurrent regimes. These results complement the  fluctuations obtained  at fixed time and the Functional Limit Theorems  obtained in the strongly transient regime.
    
    The techniques involve  original ideas of Le Gall and Rosen for fluctuations and   allow   to show tightness in some Hölder space, thus also providing sharp regularity results about  the limiting processes. 

    The original motivation of this work is the description of functionals appearing in spatial ecology for consumption of resources induced by random motion. We apply our result to estimate the large fluctuations of energy and mortality for a simple prey predator model.
\end{abstract}

\section{Introduction}

\subsection{Literature overview and motivations}

In what follows, we work on a probability space $(\Omega,\mathcal F,\p)$. Let $(X=(X_n)_{n\ge0},(\p_x)_{x\in\Z^d})$ be a $\Z^d$-valued random walk. By this, we understand that for all $x\in\Z^d$, under the probability measure $\p_x$ and for all $n\ge0$, $X_n$ may be written as :
$$X_n=x+Y_1+\cdots+Y_n,$$
where the $(Y_i)_{i\ge1}$ are $i.i.d.$ $\Z^d$-valued random variables with law $\mu$. We use the short-hand notation $\p:=\p_0$, and when referring to the random walk we shall simply write $X$. It is a Markov chain with transition distribution $p(x,y)=\mu(y-x)$. Its characteristic function $\phi$ is denoted 
$$\phi(x):=\E[\exp(\i\langle x,X_1\rangle],\quad x\in\mathbb T^d:=2\pi\R^d/\Z^d.$$

Given a discrete or continuous-time process $H$ and two times $a<b$ (in $\N$ or $\R_+$), we denote $H(a,b)$ the sample path of $H$ between times $a$ and $b$ included. We define $R_n$ the range of $X$ up to time $n$ as the cardinality of the set $X(0,n)$, which we will note $$R_n:=|X(0,n)|.$$

Asymptotic properties of the process $(R_n)_{n\ge0}$ under appropriate scaling have been extensively studied since the works of Dvoretsky and Erdös (\cite{DvoretzkyErdos1951}), who proved that $(R_n)_{n\ge0}$ satisfies a strong law of large numbers when $X$ is the simple random walk in dimension $d\ge2$. The scale factor for $d\ge3$ was found to be $n$, while for $d=2$ it was $n/\log n$. Kesten, Spitzer and Whitman in \cite{Spitzer1965} later showed showed that $n^{-1}R_n$ converges $a.s.$ to $p:=\p_0(\forall n\ge1,\text{ }X_n\neq0)$ for any random walk. The one-dimensional case is degenerate in the sense that the strong law doesn't hold. Instead, as a consequence of Donsker's invariance principle for $L^2$ random walks, it was shown by Jain and Pruitt (\cite{JainPruitt1972}) that the scaled range converges in distribution to a continuous functional of Brownian motion.

The question of a CLT for the range was first tackled by Jain and Orey in \cite{JainOrey1968}, who showed that $\Var(R_n)^{-1/2}(R_n-\E[R_n])$ converges in law to a standard Gaussian distribution in the case when $X$ is stongly transient and $p<1$. This result was extended to the transient case under some second moment assumptions in $d=2,3$ by Jain and Pruitt (\cite{JainPruitt1971}). It is interesting to note that in dimensions $d\ge4$, we have $\Var(R_n)\sim c_dn$ for some constant $c_d$, and for $d=3$, $\Var(R_n)\sim c_3 n\log n$. The case $d=2$ was sorted by Le Gall (\cite{LeGallIntersectionsMarchesAléatoires1986}), where this time the fluctuations were found to be non-Gaussian and are given by the self-intersection local time of planar Brownian motion when the walk considered is in $L^2$. In this case, we have $\Var(R_n)\sim c_2n^2/(\log n)^4$. For an accurate recount of such results and further results on strong invariance principles, large deviations and laws of iterated logarithms, one can consult \cite{BassChenRosen2006}.\\

Access to resources is crucial for survival and reproduction for any living organisms, from microbes to large animals or plants. It has in particular been shown that starvation can delay maturation, lower reproduction or even lead individuals to death (e.g. \cite{zhang2015effects},  \cite{gergs2014body},  \cite{monaco2014dynamic}). Several ecological theories have formalized this idea at different scales, from the growth and ontogeny of single individuals (e.g. \cite{west2001general},  \cite{kooijman2010dynamic},  \cite{filgueira2011comparison}, \cite{sousa2010dynamic}) to the dynamics of populations, communities and ecosystems (e.g. \cite{vanoverbeke2008modeling}, \cite{brown2004toward}). Even though some of these models consider individuals as the natural unit to consider, all of them are deterministic. One of the main idea behind these theories is that individuals acquire resources at a given rate, through predation for instance, while they consume it at another rate for their maintenance, growth, maturation and reproduction. However, the rate of resources acquisition is typically supposed constant, or linearly dependent on the condition of the individual, and does not emerge neither from the individual behaviour and its current state themselves, nor from its interaction with the resource or the prey. This has some important theoretical and applied limitation, for example for species conservation, as it is not clear how the deterministic models mentioned above effectively capture the mechanisms underlying resource acquisition and use by individuals.

Effective mortality and natality at a given time $n$ depend on the total resources consumed up to that time, while accounting for the temporal dissipation of their impact: the further in the past a resource was consumed, the less influence it is expected to have, beyond a certain latency time. We propose here a simple model in which the impact of resource consumption on survival is summarized by a subadditive functional of the trajectory of the predator. We introduce a decreasing function $m$ that quantifies the contribution at time $t$ of a prey consumed at time $s$ as $m(t-s)$. Assuming no resource regeneration and that initially one prey is available on each site, the prey consumed by a random walker corresponds to the times at which new sites are discovered—in other words, the increments of the walker’s range. A natural generalization of this model is one where preys are rather placed on a percolation cluster of $\Z^d$ (see \cite{bansaye:hal-04595815}). The cumulative effect of resource consumption at time $t$ is then given by :
\begin{equation}
    \label{eq:definition of energy}
    E_t:=\sum_{\tau_i\le t}m(t-\tau_i)=\int_0^tm(t-s)\,\d R_s,
\end{equation}
where $R_n$ denotes the range of the predator at time $n$ and $\tau_i:=\inf\{n\ge1\mid R_n=i\}$. The goal of this paper is to establish some CLT for the process $(E_t)_{t\ge0}$, which we refer to as the \textit{energy}, under some regularity assumptions on $m$ and on the motion of the predator. We now give some description of the setting in which we shall be working.

\subsection{Main results and strategy of proof}

We wish to establish functional results concerning stable random walks, motivated in particular by population dynamics, which notably include the $L^2$ case. For all $\beta\in(0,2]$, we let $U^\beta:=(U_t^\beta)_{t\ge0}$ be a non-degenerate strictly stable process of index $\beta$ in $\R^d$, that is a Lévy process such that for all $t>0$, $t^{1/\beta}U_1^\beta\overset{\mathrm{(d)}}=U_t^\beta$ (strict $\beta$-stability) and such that the law of $U_1^\beta$ isn't supported on a strict subspace of $\R^d$ (non-degeneracy). When $d=1$, we shall suppose that $U^\beta$ isn't a stable subordinator. We introduce the following assumptions :
\begin{itemize}
    \setlength\itemsep{0em}
    \item[(A1)] The additive subgroup $G$ generated by $\{x\in\Z^d\mid\p(Y_1=x)>0\}$ is $\Z^d$.
    \item[(A2)] $X$ is in the domain of attraction of a stable law $i.e$, that there is some $\beta\in(0,2]$ and a strictly increasing continuous function $b_\beta$ of regular variation of index $\beta^{-1}$ such that the following convergence in distribution holds :
    $$b_\beta(n)^{-1}X_n\distribution U_1^\beta.$$
    \item[(A3)] $\phi\in\mathcal C^1(\mathbb T^d\backslash\{0\})$ and for all $x\in\mathbb T^d\backslash\{0\}$,
    $$|\nabla\phi(x)|\le\frac{C}{|x|b_\beta^{-1}(|x|^{-1})}.$$
\end{itemize}
We shall constantly suppose that $X$ satisfies assumptions $(A1)$ and $(A2)$. The first assumption can be removed, but we keep it for sake of simplicity. It says that any site of $\Z^d$ can be reached by $X$ and that $X$ is aperiodic. If it is not satisfied, then we can find a group homeomorphism $\varphi$ and $p\le d$ such that $G=\varphi(\Z^p)$. If $X$ is aperiodic, we consider $\tilde X:=\varphi(X)$, and otherwise we can artificially make $X$ aperiodic by adding $0$ to the support of its transition distribution by considering for example the modified kernel 
$p'(x,y)=2^{-1}\mu(y-x)+2^{-1}\1_{\{x=y\}}$,
by letting $X'$ be the random walk with transition distribution $p'$ and considering $\tilde X:=\varphi(X')$. In both cases, $\tilde X$ is aperiodic. 

The third assumption is a useful tool to give bounds on the hitting times of $X$. Essentially, it is the analytical consequence used to bound the characteristic function of $X$ under the hypothesis that $X$ has a finite first moment (compare to \cite{LeGallIntersectionsMarchesAléatoires1986}, Lemma $3.1.$). As a matter of fact, it is automatically satisfied as soon as $X$ is integrable, which occurs as soon as $\beta>1$ (see \cite{LeGallRosenRangeOfStableRandomWalks1991}). 

The second assumption is essential, as it tells us that the rescaled random walk approaches a $\beta$-stable process and will allow us to identify the limits in distribution. In this setting, Le Gall and Rosen (\cite{LeGallRosenRangeOfStableRandomWalks1991}) proved that $(R_n)_{n\ge1}$ satisfies the strong law and CLT. The scales and limit distributions obtained depend on the value of the ratio $d/\beta$, as in the $L^2$ case (for which $\beta=2$). When $d/\beta\ge3/2$, the second-order fluctuations are Gaussian, for $1\le d/\beta<3/2$ the second-order fluctuations are non-Gaussian and given by a random variable that counts the time that $U^\beta$ spends self-intersecting up to time $t=1$ (see Section \ref{section:Intersection Local Times of Stable Processes in R^d} for more details and an explanation of this phenomenon) and for $d/\beta<1$, the CLT does not hold and a slightly different convergence in distribution takes place. Note that in the case of $L^2$ walks where the scale limit of $X$ is none other that the $d$-dimensional Brownian motion, $i.e.$ $U^2$, these cases correspond respectively to the dimensions $d\ge3$, $d=2$ and $d=1$. The fact that $b_\beta$ is strictly increasing and continuous is important but can be supposed without loss of generality (see \cite{feller1991introduction}, Chapter XVII.5).

Examining the form of the process $(E_t)_{t\ge0}$ defined in (\ref{eq:definition of energy}), we clearly see that the existing CLT for $(R_n)_{n\ge1}$ isn't sufficient to establish a CLT for $(E_t)_{t\ge0}$, since $E_t$ is a function of the trajectory $(R_n)_{n\le t}$. Naturally, this incites us to obtain a functional version of the aforementioned CLT, $i.e.$ convergence of the process $\left(\Var(R_n)^{-1/2}(R_{\lfloor nt\rfloor}-\E[R_{\lfloor nt\rfloor}]\right)_{t\ge0}$. Such an FCLT was obtained by Cygan, Sandrić and Šebek (\cite{CyganSandricSebekFCLTCapacityAndRange2019}), who showed that when $d/\beta>3/2$, the following convergence in distribution holds in the $J_1$ topology :
\begin{equation}
\label{distribution:FCLT for d/beta>3/2}
    \left(n^{-1/2}(R_{\lfloor nt\rfloor}-\E[R_{\lfloor nt\rfloor}])\right)_{t\ge0}\distribution(W_{Ct})_{t\ge0}.
\end{equation}
for some constant $C>0$, where $W$ denotes a standard Brownian Motion in $\R$. For applications to our model, this result isn't quite sufficient as it doesn't include the natural case of the simple symmetric planar random walk. Also, as we shall see, convergence in the $J_1$ topology is too weak for our purposes. Therefore, we wish to extend (\ref{distribution:FCLT for d/beta>3/2}) to the regime $d/\beta\le3/2$ and prove a slightly stronger form of convergence. 

We introduce the linearly interpolated version of the range 
\begin{equation*}
    \mathcal R_t:=R_{\lfloor t\rfloor}+(R_{\lfloor t\rfloor+1}-R_{\lfloor t\rfloor})(t-\lfloor t\rfloor),\quad t\ge0,
\end{equation*}
as well as the functions defined for $n\ge1$ as 
\begin{equation*}
    h(n):=\sum_{k=1}^n\p(X_k=X_0),\quad g(n):=\sum_{k=1}^nk^2b_\beta(k)^{-2d}.
\end{equation*}
$h(n)$ is the Green function of $X$ evaluated at $(0,0)$ and truncated at $n$, and $g$ is a scale function that appears naturally when calculating estimates for the variance of the range (see \cite{LeGallRosenRangeOfStableRandomWalks1991}). For example, in the case in the case where  the limiting stable process is an isotropic Brownian motion ($\beta=2$, $b_\beta(x)=\sqrt x$) with covariance $K(s,t):=\sigma^2(t\wedge s)$, $\sigma>0$, then it is well known on one hand that if $d=2$, we have
\begin{equation*}
    h(n)\underset{n\rightarrow\infty}\sim \frac{\log n}{2\pi\sigma}
\end{equation*}
and if $d=3$, then
\begin{equation*}
    g(n)=\sum_{k=1}^n k^{-1}\underset{n\rightarrow\infty}\sim\ln n.
\end{equation*}
Our first main result is the following FCLT for the continuous process $(\mathcal R_t)_{t\ge0}$, the proof of which is found in Section \ref{section:FCLT}.
\begin{theorem}
\label{theorem:main result}
    Under the assumptions $\mathrm{(A1)}$ and $\mathrm{(A2)}$, we have the following convergences in distribution on $\mathcal C(\R_+)$ endowed with the topology of uniform convergence on compacts
    \begin{itemize}
    \setlength\itemsep{0em}
        \item If $d/\beta\ge3/2$, there is a constant $\sigma^2>0$ such that
        \begin{equation}
            \label{convergence:FCLT for d/beta=3/2}
            \left((ng(n))^{-1/2}\left(\mathcal R_{ nt}-\E[\mathcal R_{ nt}]\right)\right)_{t\ge0}\distribution \left(W_{\sigma^2t}\right)_{t\ge0},
        \end{equation}
        where $W$ is a standard $1$-dimensional Brownian motion.
        \item If $1\le d/\beta<3/2$ then under assumption $(A3)$, we have 
        \begin{equation}
        \label{convergence:FCLT for 1<=d/beta<3/2}
            \left(\frac{h(n)^2b_\beta(n)^d}{n^2}\left(\mathcal R_{ nt}-\E[\mathcal R_{ nt}]\right)\right)_{t\ge0}\distribution\left(-\gamma_t^\beta\right)_{t\ge0},
        \end{equation}
        where $(\gamma_t^\beta)_{t\ge0}$ is the renormalized self-intersection local time of $U^\beta$ (see Section \ref{section:Intersection Local Times of Stable Processes in R^d} for a definition).
        \item If $d/\beta<1$,
        \begin{equation}
            \label{convergence:FLT for d/beta<1}
            \left(b_\beta(n)^{-1}\mathcal R_{ nt}\right)_{t\ge0}\distribution\left(\lambda\left(U^\beta(0,t)\right)\right)_{t\ge0},
        \end{equation}
    where $\lambda$ denotes the Lebesgue measure on $\R$ (note that $d/\beta<1\implies d=1)$.
    \end{itemize}
\end{theorem}
Note that if $d/\beta>3/2$, $g$ is bounded and so we recover (\ref{distribution:FCLT for d/beta>3/2}). We recall the definition of $(E_t)_{t\ge0}$ and introduce the analogous process $(\mathcal E_t)_{t\ge0}$ :
\begin{equation*}
    E_t=\int_0^tm(t-s)\,\d R_s,\quad\mathcal E_t:=\int_0^tm(t-s)\,\d\mathcal R_s.
\end{equation*}
As an application of Theorem \ref{theorem:main result}, we get our second main result, the proof of which may be found in Section \ref{section:LT energy}. It shall be extended to the functional setting in Section \ref{section:FLT energy}.
\begin{theorem}
\label{theorem:main result 2}
    Suppose that $m$ is continuously differentiable, of regular variation of index $\chi$ and has a monotone derivative. Then under assumptions $(A1)$ and $(A2)$, for all $t\ge0$, the following convergences in distribution hold :
    \begin{itemize}
        \item If $d/\beta\ge3/2$ and $\chi>-1/2$,
        \begin{equation}
            \frac{1}{m(n)\sqrt{ng(n)}}\left(E_{nt}-\E[E_{nt}]\right)\distribution\sigma\int_0^t(t-s)^\chi\,\d W_{s}
        \end{equation}
        where $\sigma^2$ is as in Theorem \ref{theorem:main result}, 
        \item If $1\le d/\beta<3/2$ and $\chi>\beta/d-2$, under the additional hypothesis $(A3)$,
        \begin{equation}
            \frac{h(n)^2b_\beta(n)^d}{m(n)n^2}\left(E_{nt}-\E[E_{nt}]\right)\distribution -\int_0^t(t-s)^\chi\,\d\gamma_s^\beta,
        \end{equation}
        \item If $d/\beta<1$ and $\chi>-1/\beta$,
        \begin{equation}
            \frac{1}{m(n)b_\beta(n)}E_{nt}\distribution\int_0^t(t-s)^\chi\,\d L(s),
        \end{equation}
        where $L(s):=\lambda\left(U^\beta(0,s)\right)$ for all $s\ge0$.
    \end{itemize}
    We interpret all of limiting integrals in the sense of Young (see Appendix \ref{appendix:Young integral}).
\end{theorem}
In our setting, we wish $m$ to be continuously decreasing and such that $m(t)$ goes to $0$ as $t\rightarrow\infty$. Such examples are given by $m(t):=L/(1+t)^\delta$ for some $L,\delta>0$. These choices satisfy our regularity hypotheses, so long as $\delta$ isn't too large. We now briefly discuss our strategy of proof. We let 
\begin{equation*}
    \overline R_t^{(n)}=S(n)(R_{\lfloor nt\rfloor}-\E[R_{\lfloor nt\rfloor}]),\quad\overline{\mathcal R}_t^{(n)}=S(n)(\mathcal R_{nt}-\E[\mathcal R_{nt}]),
\end{equation*}
where the function $S(n)$ is an equivalent of $\Var(R_n)^{-1/2}$ (which thus depends on the value of the ratio $d/\beta$). 
All we need to prove Theorem \ref{theorem:main result} is to show tightness and convergence of finite-dimensional marginals of $\overline{\mathcal R}^{(n)}$. For the convergence of finite dimensional marginals, we use the Cramér-Wold Theorem as in \cite{CyganSandricSebekFCLTCapacityAndRange2019}. Since $\overline{\mathcal R}^{(n)}$ is simply the continuous linearly interpolated version of $\overline R^{(n)}$, it is sufficient to show the convergence of the marginals of the latter. In order to exhibit tightness, we use Kolmogorov's criterion for $\overline{\mathcal R}^{(n)}$. The key to showing that this criterion applies is obtaining sharp bounds on the normalized moments of the range, which we carry out in each corresponding section, and also on the normalized moments of the joint range of two independent random walks, which corresponds to Lemma \ref{lemma:Hölder bound for moments of I}. A byproduct of this approach is uniform in $n$ local Hölder continuity of the $\overline{\mathcal R}^{(n)}$ and of the limiting process, which we may use to prove Theorem \ref{theorem:main result 2} for $\mathcal E$. Since $|\mathcal E_t-E_t|\le1$, we deduce the announced version of Theorem \ref{theorem:main result 2}. Note also that we could rewrite Theorem \ref{theorem:main result} by replacing $\mathcal R$ with $R$, and the convergence would hold in the $J_1$ topology instead of the topology of uniform convergence on compacts for continuous functions.

\section{Preliminary results and notations}
\label{section:Preliminary results and notations}

We say that a positive, measurable, real valued function $f$ is of \textit{regular variation of index $\kappa$} if for all $c>0$, we have 
$$\underset{x\rightarrow\infty}\lim\frac{f(cx)}{f(x)}=c^\kappa.$$
In particular, any function $f$ of regular variation of index $\kappa$ may be written as $f(x):=x^\kappa g(x)$ for some positive measurable function $g$ such that for all $c>0$, 
$$\underset{x\rightarrow\infty}\lim\frac{g(cx)}{g(x)}=1.$$
Therefore, $g$ is of regular variation of index $0$. Such functions shall be referred to as \textit{slowly varying functions}. As it turns out, a function $f$ is of regular variation of index $\kappa$ if and only if it may be written in the form $f(x)=x^\kappa g(x)$ for some slowly varying function $g$. As described in \cite{BinghamGoldieTeugelsRegularVariation1987}, for $f$ a function of regular variation of index $\kappa$ and $K$ a compact subset of $(0,\infty)$, we have 
\begin{equation}
        \label{lim:uniform convergence for reguarly varying functions}
        \underset{x\rightarrow\infty}{\lim}\underset{c\in K}\sup\left|\frac{f(cx)}{f(x)}-c^\kappa\right|=0.
\end{equation}
This property allows to deduce the following useful result :
\begin{lemma}
    \label{lemma:limit of f(x_n)/f(y_n) (regular variation)}
    Let $f$ be of regular variation of index $\kappa$ and $(x_n)_{n\ge1}$, $(y_n)_{n\ge1}$ be two positive real sequences such that $\underset{n\rightarrow\infty}\lim x_n=\underset{n\rightarrow\infty}\lim y_n=\infty$ and $x_n/y_n\underset{n\rightarrow\infty}\sim\ell$ for some $\ell>0$. Then
    \begin{equation}
    \label{conv:limit of f(x_n)/f(y_n) (regular variation)}
        \underset{n\rightarrow\infty}\lim\frac{f(x_n)}{f(y_n)}=\ell^\kappa.
    \end{equation}
\end{lemma}
\begin{proof}
    Let $0<\varepsilon<\ell$. Then for $n$ sufficiently large, we have $x_n/y_n\in K_{\ell,\varepsilon}:=[\ell-\varepsilon,\ell+\varepsilon]$. Hence 
    $$\underset{n\rightarrow\infty}\lim\left|\frac{f(x_n)}{f(y_n)}-\left(\frac{x_n}{y_n}\right)^\kappa\right|=\underset{n\rightarrow\infty}\lim\left|\frac{f\left(\frac{x_n}{y_n}\cdot y_n\right)}{f(y_n)}-\left(\frac{x_n}{y_n}\right)^\kappa\right|\le\underset{n\rightarrow\infty}\lim\underset{c\in K_{\ell,\varepsilon}}\sup\left|\frac{f(cy_n)}{f(y_n)}-c^\kappa\right|=0$$
    and the conclusion is reached.
\end{proof}
We shall also make use of the so-called Potter bounds for regularly varying functions (see \cite{BinghamGoldieTeugelsRegularVariation1987}). If $f$ is a regularly varying function of index $\kappa$ that is also bounded away from $0$ and $\infty$ on all compact subsets of $[0,\infty)$, then for all $\varepsilon>0$, there is some constant $C_\varepsilon>0$ such that for all $x,y\in(0,\infty)$,
\begin{equation}
\label{ineq:Potter bounds for regularly varying functions}
    C_\varepsilon^{-1}\left(\frac{x}{y}\right)^{\kappa-\varepsilon}\wedge\left(\frac xy\right)^{\kappa+\varepsilon}\le\frac{f(x)}{f(y)}\le C_\varepsilon\left(\frac{x}{y}\right)^{\kappa-\varepsilon}\vee\left(\frac xy\right)^{\kappa+\varepsilon}
\end{equation}

We define $l_\beta$ as the continuous increasing inverse of $b_\beta$. Also, we let $s_\beta$ be the slowly varying function such that $b_\beta(x)=x^{1/\beta}s_\beta(x)$. We now mention the following fact that will be useful in the case $d/\beta\ge3/2$.
    \begin{equation}
    \label{lim:technical limit with s and g}
        \underset{n\rightarrow\infty}\lim s_\beta(n)^d\sqrt{g(n)}=\infty.
    \end{equation}
Indeed, by taking $p\ge1$ and $\eta\in(0,1)$, we have
\begin{equation}
    \underset{n\rightarrow\infty}\lim\underset{\lceil n/p\rceil\le k\le n}\max\left|\frac{s_\beta(n)}{s_\beta(k)}-1\right|\le\underset{n\rightarrow\infty}\lim\underset{c\in[1/p,1]}\sup\left|\frac{s_\beta(n)}{s_\beta(cn)}-1\right|=0
\end{equation}
whence for $n$ sufficiently large, $\underset{\lceil n/p\rceil\le k\le n}\min\text{ }s_\beta(n)/s_\beta(k)\ge1-\eta$, and in turn
\begin{equation}
    g(n)\ge\frac{1}{s_\beta(n)^{2d}}\sum_{k=\lceil n/p\rceil}^n\frac{s_\beta(n)^{2d}}{ks_\beta(k)^{2d}}\ge\frac{(1-\eta)^{2d}}{s_\beta(n)^{2d}}\sum_{k=\lceil n/p\rceil}^n\frac1k.
\end{equation}
Taking limits on both sides yields $\underset{n\rightarrow\infty}\liminf g(n)s_\beta(n)^{2d}\ge(1-\eta)^{2d}\log p$ and the conclusion is reached since $p$ is arbitrary. Applying a similar line of reasoning, we may also show in this case that $g$ is slowly varying.

We now turn to some properties on the number of common points in the range of two independent random walks. In what follows, we let $X'$ denote an $i.i.d.$ version of $X$, and for $n,m\ge1$ we introduce the quantity
\begin{equation*}
    I_{n,m}:=|X(0,n)\cap X'(0,m)|,\quad I_{n,0}=I_{0,m}=0.
\end{equation*} 
The choice of convention when one of the arguments is zero may seem arbitrary, but it doesn't change any of the results. Indeed, in practice we shall either be examining the quantity $I_{\lfloor ns\rfloor,\lfloor nt\rfloor}$ at fixed times $s$ and $t$ and letting $n$ go to infinity, in which case our convention isn't important, or we shall be using $I_{n,m-n}$ to represent $|X(0,n)\cap X(n+1,m)|$ as in the first line of the proof of Lemma \ref{lemma:Hölder bound moments of scaled range 1<=d/beta<3/2} for example. In this case, if $m-n=0$, then by convention $|X(0,n)\cap X(n+1,m)|=0$. Either way, we may thus add this slight modification to the definition. When $Z$ is an integrable, we employ the notation
\begin{equation*}
    \{Z\}:=Z-\E[Z].
\end{equation*}
We shall be using the following inequality \cite{LeGallRosenRangeOfStableRandomWalks1991} :
\begin{equation}
    \label{ineq:bound for centered 2p-th moment of intersections}
    \E[(I_{n,m})^p]\le (p!)^2\E[I_{n,m}]^p,\quad n,m,p\ge1.
\end{equation}
In the case $n=m$, it is Lemma $3.1$ of \cite{LeGallRosenRangeOfStableRandomWalks1991}, and we see that identical arguments given in the proof carry out to the case $n\ne m$. As a consequence, we obtain the following 
\begin{equation}
    \label{ineq:bound for centered moments of intersections}
    \E\left[\left\{I_{n,m}\right\}^{2p}\right]\le C_p\left(\E\left[(I_{n,m})^{2p}\right]+\E[I_{n,m}]^{2p}\right)\le (p!)^2C_p\E[I_{n,m}]^{2p}.
\end{equation}
For some constant $C_p>0$ depending only on $p$. As we shall see in section \ref{section:FCLT}, it is crucial to obtain some Hölder-type estimate for the scaled moments of $(s,t)\mapsto I_{\lfloor ns\rfloor,\lfloor nt\rfloor}$. To do so, we use the following Lemma.
\begin{lemma}
\label{lemma:Hölder bound for moments of I}
    Let $n\ge1$, $T>0$ and $s,t\in(0,T]$. Then for any $\eta>0$ sufficiently small, there is a constant $C_{\eta}>0$ such that 
    \begin{equation}
        \label{ineq:Hölder bounds for moments of I_{ns,nt}}
        S_{d,\beta}(n)\E\left[I_{\lfloor ns\rfloor,\lfloor nt\rfloor}\right]\le C_\eta (s\wedge t)^{\chi_{d,\beta}-\eta},
    \end{equation}
    where 
    \begin{equation*}
        \chi_{d,\beta}:=
        \begin{cases}
            1/\beta & \text{if } d/\beta<1\\
            2-d/\beta & \text{if } 1\le d/\beta<3/2\\
            1/2 & \text{if } 3/2\le d/\beta<2\\
        \end{cases},\quad
        S_{d,\beta}(n):=
        \begin{cases}
            b_\beta(n)^{-1} & \text{if } d/\beta<1\\
            h(n)^2b_\beta(n)^d/n^2 & \text{if } 1\le d/\beta<3/2\\
            (ng(n))^{-1/2} & \text{if } 3/2\le d/\beta<2
        \end{cases}
    \end{equation*}
\end{lemma}
\begin{proof}
    We start by introducing the discrete intersection local time 
    \begin{equation*}
        J_{n,m}=\sum_{i=0}^n\sum_{j=0}^m\1_{\{X_i=X_j'\}}=\sum_{y\in\Z^d}\sum_{i=0}^n\sum_{j=0}^m\1_{\{X_i=y\}}\1_{\{X_j'=y\}}.
    \end{equation*}
    Reasoning as in \cite{LeGallRosenRangeOfStableRandomWalks1991}, we first observe the following obvious inequality
    \begin{equation}
    \label{ineq:bound 1 for E[I_n] wrt to E[J_n]}
        \sum_{y\in\Z^d}\E\left[\1_{\{y\in X(0,n)\}}\1_{\{y\in X'(0,m)\}}\sum_{i=0}^{2n}\1_{\{X_i=y\}}\sum_{i=0}^{2m}\1_{\{X_j'=y\}}\right]\le \E[J_{2n,2m}].
    \end{equation}
    Then, by letting $T_y:=\inf\{n\ge0\mid X_n=y\}$ for $y\in\Z^d$, we use the Markov property at time $T_y$ to get the following bound
    \begin{align}
    \label{ineq:bound 2 for E[I_n] wrt to E[J_n]}
        \E\left[\1_{\{y\in X(0,n)\}}\sum_{i=1}^{2n}\1_{\{X_i=y\}}\right]
        &=\E\left[\sum_{\ell=1}^n\1_{\{T_y=\ell\}}\sum_{i=\ell}^{2n}\p_y\left(X_{i-\ell}=y\right)\right]\nonumber\\
        &\ge\p\left(y\in X(0,n)\right)\sum_{i=0}^n\p_y(X_i=y)=\p(y\in X(0,n))h(n),
    \end{align}
    where the last line is due to the fact that $2n-\ell\ge n$ and that $X$ is translation invariant. Summing (\ref{ineq:bound 2 for E[I_n] wrt to E[J_n]}) over $y\in\Z^d$, combining with (\ref{ineq:bound 1 for E[I_n] wrt to E[J_n]}) and invoking independence, we obtain
    \begin{equation}
        h(n)h(m)\E[I_{n,m}]\le \E[J_{2n,2m}].
    \end{equation}
    Therefore for $\eta>0$ sufficiently small, by taking $0\le s\le t\le T$ and by noting that $I_{\lfloor ns\rfloor,\lfloor nt\rfloor}\le I_{\lfloor ns\rfloor,\lfloor nT\rfloor}$,
    \begin{align}
    \label{ineq:control of E[I_{ns,nt}] by E[J_{ns,nt}]}
        \frac{h(n)^2b_\beta(n)^d}{n^2}\E\left[I_{\lfloor ns\rfloor,\lfloor nt\rfloor}\right]&\le\frac{h(n)^2}{h(\lfloor ns\rfloor)h(\lfloor nT\rfloor)}\frac{b_\beta(n)^d}{n^2}\E\left[J_{2\lfloor ns\rfloor,2\lfloor nT\rfloor}\right]\nonumber\\
        &\le C_\eta\left((sT)^{\eta}+(sT)^{-\eta}\right)\frac{b_\beta(n)^d}{n^2}\E\left[J_{2\lfloor ns\rfloor,2\lfloor nT\rfloor}\right]
    \end{align}
    where we used the Potter bounds and the fact that $h$ is slowly varying. Furthermore, by letting $\mathbb T_n^d:=b_\beta(n)\mathbb T^d$, we have
    \begin{align}
    \label{ineq:bound moment discrete intersection local time}
        \frac{b_\beta(n)^d}{n^2}\E\left[J_{\lfloor ns\rfloor,\lfloor nt\rfloor}\right]&=C\E\left[\frac{1}{n^2}\sum_{i=0}^{\lfloor ns\rfloor}\sum_{j=0}^{\lfloor nt\rfloor}\int_{\mathbb T_n^d}\,\d z\exp\left(\i\left\langle\frac{z}{b_\beta(n)},X_i-X_j'\right\rangle\right)\right]\nonumber\\
        &=C\int_{\mathbb T_n^d}\,\d z\frac{1}{n^2}\sum_{i=0}^{\lfloor ns\rfloor}\sum_{j=0}^{\lfloor nt\rfloor}\phi\left(\frac{z}{b_\beta(n)}\right)^i\phi\left(-\frac{z}{b_\beta(n)}\right)^j.
    \end{align}
    
    \textit{Case 1:} If $1\le d/\beta<3/2$, fix $\eta>0$ sufficiently small and take $p_1,q_1,p_2,q_2\ge1$ such that
    \begin{equation*}
        q_1:=\frac{\beta}{d-\beta+\eta\beta},\quad q_2<\frac{1}{1-\eta},\quad p_k^{-1}+q_k^{-1}=1\text{ $(k=1,2$)}.
    \end{equation*}
    By separating the sums and using two Hölder inequalities, the integrand in (\ref{ineq:bound moment discrete intersection local time}) may be bounded by 
    \begin{equation}
    \label{bound:Hölder inequality bound for moment of discrete local time}
        \left(\frac{\lfloor ns\rfloor}{n}\right)^{1/p_1}\left(\frac{\lfloor nt\rfloor}{n}\right)^{1/p_2}\left(\frac{1}{n}\sum_{i=0}^{\lfloor nT\rfloor}\left|\phi\left(\frac{z}{b_\beta(n)}\right)\right|^{q_1i}\right)^{1/q_1}\left(\frac{1}{n}\sum_{j=0}^{\lfloor nT\rfloor}\left|\phi\left(\frac{z}{b_\beta(n)}\right)\right|^{q_2i}\right)^{1/q_2}.
    \end{equation}
    By $(5.15)$ in \cite{RosenRandomWalks&IntersectionLocalTimes1990}, for any $\varepsilon>0$ sufficiently small and $z\in\R^d$, there is some $C_{\varepsilon,T}>0$ such that for any $q\ge1$,
    \begin{equation*}
        \left(\frac{1}{n}\sum_{i=0}^{\lfloor nT\rfloor}\left|\phi\left(\frac{z}{b_\beta(n)}\right)\right|^{qi}\right)^{1/q}\le\frac{C_{\varepsilon,T}}{1+|\overline z|^{(\beta-\varepsilon)q^{-1}}}.
    \end{equation*}
    where $\overline z$ is the representative of $z$ modulo $2\pi b_\beta(n)$. Consequently, (\ref{bound:Hölder inequality bound for moment of discrete local time}) is bounded by
    \begin{equation}
    \label{bound:Hölder inequality bound for moment of discrete local time 2}
        C_{\varepsilon,T}s^{1/p_1}t^{1/p_2}\frac{1}{(1+|\overline z|^{(\beta-\varepsilon)q_1^{-1}})(1+|\overline z|^{(\beta-\varepsilon)q_2^{-1}})}.
    \end{equation}
    Plugging (\ref{bound:Hölder inequality bound for moment of discrete local time 2}) into (\ref{ineq:bound moment discrete intersection local time}), we get by plainly bounding $t$ by $T$,
    \begin{equation}
    \label{ineq:bound moment discrete intersection local time 2}
        \frac{b_\beta(n)^d}{n^2}\E\left[J_{\lfloor ns\rfloor,\lfloor nt\rfloor}\right]\le C_{\varepsilon,T,\eta}s^{1/p_1},
    \end{equation}
    where $C_{\varepsilon,T,\eta}$ is a finite constant that doesn't depend on $n$ so long as $(\beta-\varepsilon)(q_1^{-1}+q_2^{-1})>d$. One easily checks that this condition is satisfied with our choice of $q_1,q_2$ if $\varepsilon$ is taken sufficiently small, and that our choice of $q_1$ implies $1/p_1=2-d/\beta-\eta$, whence combining (\ref{ineq:control of E[I_{ns,nt}] by E[J_{ns,nt}]}) and (\ref{ineq:bound moment discrete intersection local time 2}) is enough to conclude.

    \textit{Case 2:} If $3/2< d/\beta<2$, we once again see that our previous choice of $p_1,q_1,p_2,q_2$ remains valid. Therefore, we have 
    \begin{equation*}
        \frac{h(n)^2b_\beta(n)^d}{n^2}\E\left[I_{\lfloor ns\rfloor,\lfloor nt\rfloor}\right]\le C_\eta s^{2-d/\beta-\eta}.
    \end{equation*}
    Writing the inverse of the scale function as $n^{2-d/\beta}s_\beta(n)^{-d}h(n)^{-2}$, noticing that $h(n)^{-2}=O(1)$ since this regime implies that $X$ is transient and applying the Potter bounds to the slowly varying function $n\mapsto s_\beta(n)^{d}$, we get 
    \begin{align}
    \label{ineq:bound scaled moment intersection d/beta>3/2}
        n^{-1/2}\E\left[I_{\lfloor ns\rfloor,\lfloor nt\rfloor}\right]&=\frac{n^{3/2-d/\beta}}{s_\beta(n)^{d}h(n)^2}\frac{h(n)^2b_\beta(n)^d}{n^2}\E\left[I_{\lfloor ns\rfloor,\lfloor nt\rfloor}\right]\nonumber\\
        &\le C_\eta n^{3/2-d/\beta+\eta}s^{2-d/\beta-\eta}.
    \end{align}
    By the convention we made on $I_{n,m}$, if $s<1/n$ then $\lfloor ns\rfloor=0$ and $I_{\lfloor ns\rfloor,\lfloor nt\rfloor}=0$, whence the inequality we are trying to prove is trivially true. Therefore, we suppose that $s\ge1/n$ which yields $n^{3/2-d/\beta+\eta}\le s^{d/\beta-3/2-\eta}$ so long as $\eta<d/\beta-3/2$. Plugging this into (\ref{ineq:bound scaled moment intersection d/beta>3/2}) yields the expected result, since in this regime we also have $g(n)^{-1/2}=O(1)$.

    \textit{Case 3:} The case $d/\beta=3/2$ is essentially contained in the previous one, except that this time we have 
    \begin{equation*}
        (ng(n))^{-1/2}\E\left[I_{\lfloor ns\rfloor,\lfloor nt\rfloor}\right]\le C_\eta s_\beta(n)^{-d}g(n)^{-1/2} s^{1/2-\eta}.
    \end{equation*}
    By (\ref{lim:technical limit with s and g}), $s_\beta(n)^{-d}g(n)^{-1/2}=O(1)$ and so we conclude as previously.

    \textit{Case 4:} We now finally examine the case $d/\beta<1$, which implies $d=1<\beta$. We notice that the choice of $p_1,q_1,p_2,q_2$ from \textit{Case 1} is still valid so long as $\eta\in(1-1/\beta,1)$. Fixing such an $\eta$ that is close to $1-1/\beta$, letting $\tilde \eta:=\eta-1+\beta>0$, choosing $\eta=\tilde\eta+1-1/\beta$ in (\ref{ineq:control of E[I_{ns,nt}] by E[J_{ns,nt}]}) and then reasoning exactly as in \textit{Case 1} yields
    \begin{equation*}
        \frac{h(n)^2b_\beta(n)}{n^2}\E\left[I_{\lfloor ns\rfloor,\lfloor nt\rfloor}\right]\le C_{\eta,T}(s\wedge t)^{2-d/\beta-2\eta}=C_{\tilde \eta,T}(s\wedge t)^{1/\beta-2\tilde \eta}.
    \end{equation*}
    We must now show that
    \begin{equation}
    \label{ineq:scale function is O(1) d/beta<1}
        \frac{S_{d,\beta}(n)n^2}{h(n)^2b_\beta(n)}=O(1).
    \end{equation}
    By $(2.j)$, \cite{LeGallRosenRangeOfStableRandomWalks1991}, we have for some constant $C>0$,
    \begin{equation*}
        h(n)\underset{n\rightarrow\infty}\sim C\sum_{k=1}^nb_\beta(k)^{-1}.
    \end{equation*}
    Using the Potter bounds we get for any $\varepsilon>0$,
    \begin{equation*}
        \sum_{k=1}^n\frac{b_\beta(n)}{b_\beta(k)}\ge C_\varepsilon\sum_{k=1}^n\left(\frac{n}{k}\right)^{1/\beta-\varepsilon}\underset{n\rightarrow\infty}\sim \frac{C_\varepsilon}{1-1/\beta+\varepsilon} n.
    \end{equation*}
    Thus,
    \begin{equation*}
        \frac{n}{b_\beta(n)}=O(h(n))\implies \frac{n}{h(n)b_\beta(n)}=O(1).
    \end{equation*}
    Recalling that $S_{d,\beta}(n)=b_\beta(n)^{-1}$, this proves (\ref{ineq:scale function is O(1) d/beta<1}) and concludes the proof.
\end{proof}

\section{Renormalized Self-Intersection Local Times of Stable Processes in $\R^d$}
\label{section:Intersection Local Times of Stable Processes in R^d}

The aim of this section is to provide a simple construction of the Intersection Local Time of a stable process in $\R^d$ based on the analogous construction for planar Brownian motion performed in \cite{LeGallIntersectionsMarchesAléatoires1986}, and to explain how such a quantity naturally appears when studying the fluctuations of the range. The starting point of Le Gall in \cite{LeGallIntersectionsMarchesAléatoires1986} was to notice that the following decomposition of $R_n$ holds for all $n,p\ge1$ :
\begin{equation}
    \label{equation:decomposition of the range}
    R_n=\sum_{k=1}^p\left|X\left(\frac{k-1}{p}n,\frac kpn\right)\right|-\sum_{k=2}^p\left|X\left(0,\frac{k-1}{p}n\right)\cap X\left(\frac{k-1}{p}n,\frac kpn\right)\right|.
\end{equation}

We essentially divide $\{1,\dots,n\}$ into $p$ intervals, add the range of the walk on each of the subintervals and susbstract the intersections with the past of the walk which would be counted too many times in the range. Hence, there is a natural competition between the number of new sites visited and the number of sites visited multiple times, or in other words the amount of self-intersections of $X$.

As one can imagine, if $X$ is "sufficiently" transient it should spend more time discovering new sites than it does self-intersecting. This turns out to be the case, and so when $(R_n)_{n\ge1}$ is centered and scaled, the second term in the right-hand side of (\ref{equation:decomposition of the range}) vanishes as $n\rightarrow\infty$, and what remains is a sum of $i.i.d.$ random variables hence the fluctuations are Gaussian by the usual CLT. Roughly speaking, when the increments of $X$ aren't too large compared to $d$, then we may rather expect $X$ to spend more time self-intersecting than discovering new sites and in this case the roles of the two terms of the right-hand side of (\ref{equation:decomposition of the range}) are reversed. As $X$ is scaled and becomes $U^\beta$, one would therefore expect to obtain as $n\rightarrow\infty$ a random variable that counts the self-intersections of $U^\beta$. Such a random variable exists precisely when $1\le d/\beta<3/2$ (see \cite{Rosen1988}) and is called the \textit{Renormalized Self-Intersection Local Time of $U^\beta$}. The aim of this section is to go over a basic construction of the Renormalized Self-Intersection Local Time that we note $\gamma^\beta$, whilst highlighting some of its properties that will be important in the rest of the paper.

Let Let $\tilde U^\beta$ be an independent copy of $U^\beta$. Then for $1\le d/\beta<3/2$, one can construct a family $(\alpha^\beta(x,\cdot))_{x\in\R^{d}}$ of random Radon measures on $\R_+^2$ such that $x\mapsto\alpha^\beta(x,\cdot)$ is continuous on $\R^d$ for the vague topology of measures and for any Borel set $A$ of $\R_+^2$ and Borel function $f:\R^{d}\rightarrow\R_+$ we have the following equality :
\begin{equation}
    \label{equation:occupation density formula for intetrsection of 2 stable processes}
    \int_Af(U^\beta_{s_2}-\tilde U^\beta_{s_1})\,\d s_1\,\d s_2=\int_{\R^d}f(x)\alpha^\beta(x,A)\,\d x
\end{equation}
(see \cite{ExponentialAsymptoticsILTChenRosen2005}). The measures $(\alpha^\beta(x,\cdot))_{x\in\R^d}$ are known as the \textit{Intersection Local Times of $U^\beta$ and $\tilde U^\beta$}. Essentially, $\alpha^\beta(x,A)$ measures the time spent by $U^\beta$ and $\tilde U^\beta$ seperated by the quantity $x$ over the time set $A$, and the occupation density type formula (\ref{equation:occupation density formula for intetrsection of 2 stable processes}) allows us to formally interpret $\alpha^\beta(x,A)$ as 
$$\alpha^\beta(x,A)=\int_A\delta_{\{0\}}(U^\beta_{s_2}-\tilde U^\beta_{s_1}-x)\,\d s_1\,\d s_2.$$
 For all $t\ge0$, we note $\alpha_t^\beta:=\alpha^\beta(0,[0,t]^2)$. The scaling property of $U^\beta$ implies a scaling property for the process $\left(\alpha_t^\beta\right)_{t\ge0}$. Indeed, by considering a continuous and compactly-supported function $f$ with integral $1$ such that $f(0)>0$ and letting $f_\varepsilon(x):=\varepsilon^{d}f(\varepsilon x)$, then by (\ref{equation:occupation density formula for intetrsection of 2 stable processes}) we have for all $k,t,\varepsilon>0$, 
$$\int_{\R^d}f_\varepsilon(x)\alpha^\beta(x,[0,kt])\,\d x\overset{(\d)}=k^{2-d/\beta}\int_{\R^d}f_{\varepsilon k^{1/\beta}}(x)\alpha^\beta(x,[0,t])\,\d x,$$
and letting $\varepsilon\rightarrow0$ yields 
\begin{equation}
    \label{equation:scaling of intersection local time of 2 stable processes}
    \alpha^\beta_{kt}\overset{(\d)}=k^{2-d/\beta}\alpha^\beta_t
\end{equation}
 since $f_\varepsilon\rightarrow\delta_{\{0\}}$ in the sense of distributions.

We now wish to study self-intersections of $U^\beta$. Rosen constructs in \cite{Rosen1988} a family $(\rho^\beta(x,\cdot))_{x\in\R^d}$ of random measures such that $x\mapsto\rho^\beta(x,\cdot)$ is continuous on $\R^d\backslash\{0\}$ for the vague topology of measures, and for any Borel set $A$ of $\R_+^2$ and Borel function $f:\R^{d}\rightarrow\R_+$ we have the following equality :
\begin{equation}
    \label{equation:occupation density formula for self-intetrsection of a stable processes}
    \int_Af(U^\beta_{s_2}- U^\beta_{s_1})\,\d s_1\,\d s_2=\int_{\R^d}f(x)\rho^\beta(x,A)\,\d x.
\end{equation}

Furthermore, Rosen shows that $x\mapsto\rho^\beta(x,[0,t]^2)$ is singular at $x=0$ and determines the exact order or the singularity, which turns out to be $Kt/|x|^{d-\beta}$ for some constant $K>0$. We give an explanation for this singularity in what follows and explain how to naturally introduce a renormalized version of $(\rho^\beta(x,\cdot))_{x\in\R^d}$ which shall be denoted $(\gamma^\beta(x,\cdot))_{x\in\R^d}$ and will be referred to as the \textit{renormalized self-intersection local time of $U^\beta$}. 

For a Borel set $A$ of $\R_+^2$, we let $A_\le:=\{(x,y)\in A\mid x\le y\}$. By symmetry, instead of letting the time indices vary over $A$ we rather let them vary over $A_\le$. To lighten the notations, we let  $\rho^\beta(\cdot):=\rho^\beta(0,\cdot)$. For $t>0$ and for $j\ge1$, $1\le i\le 2^{j-1}$, we let
$$A_t^{(i,j)}:=\left[\frac{2i-2}{2^j}t,\frac{2i-1}{2^j}t\right)\times\left(\frac{2i-1}{2^j}t,\frac{2i}{2^j}t\right].$$
The reason for introducing the sets $A_t^{(i,j)}$ is firstly because $$[0,t]_{\le}^2=\bigcup_{j=1}^\infty\bigcup_{i=1}^{2^{j-1}}A_t^{(i,j)},$$ but also that from properies of $U^\beta$, we have for all $t>0$, $j\ge1$, $1\le i\le 2^{j-1}$ :
$$\left(U_{s_2}^\beta-U_{s_1}^\beta\right)_{(s_1,s_2)\in A_t^{(i,,j)}}\overset{\mathrm{(d)}}{=}\left(U_{s_2}^\beta-\tilde U_{s_1}^\beta\right)_{(s_1,s_2)\in[0,2^{-j}t]^2}$$
and so as previously, it follows from (\ref{equation:occupation density formula for intetrsection of 2 stable processes}) that 
\begin{equation}
\label{equation : scaling of self intersection local time over dyadic sets}
    \rho^\beta\left(A_t^{(i,j)}\right)\overset{\mathrm{(d)}}=2^{-j(2-d/\beta)}\alpha_t^\beta.
\end{equation}

Furthermore from the independence of increments of $U^\beta$, for $j\ge1$ and $1\le i\neq i'\le 2^{j-1}$, we have
$$\left(U_{s_2}^\beta-U_{s_1}^\beta\right)_{(s_1,s_2)\in A_t^{(i,j)}}\indep\left(U_{s_2}^\beta-U_{s_1}^\beta\right)_{(s_1,s_2)\in A_t^{(i,j)}}$$
in such a way that 
\begin{equation}
    \label{equation : independence of intersection local time over disjoint dyadic sets}\rho^\beta\left(A_t^{(i,j)}\right)\indep\rho^\beta\left(A_t^{(i',j)}\right).
\end{equation}
Introducing the sets $$A_t^{(j)}:=\bigcup_{i=1}^{2^{j-1}}A_t^{(i,j)},$$
we therefore see that on one hand, for all $j\ge1$, 
$$u_{j}^{1}:=\E\left[\rho^\beta\left(A_t^{(j)}\right)\right]\overset{\mathrm{(\ref{equation : scaling of self intersection local time over dyadic sets})}}=2^{-j(1-d/\beta)-1}\E\left[\alpha_t^\beta\right]$$
and on the other,
\begin{eqnarray*}
    u_j^2:=\E\left[\left\{\rho^\beta\left(A_t^{(j)}\right)\right\}^2\right]&\overset{\mathrm{(\ref{equation : independence of intersection local time over disjoint dyadic sets})}}=&\sum_{i=1}^{2^{j-1}}\E\left[\left\{\rho^\beta\left(A_t^{(i,j)}\right)\right\}^2\right]\overset{\mathrm{(\ref{equation : scaling of self intersection local time over dyadic sets})}}=2^{-j(3-2d/\beta)-1}\E\left[\left\{\alpha_t^\beta\right\}^2\right].
\end{eqnarray*}
The fact that $1\le d/\beta<3/2$ precisely implies that the series $\sum u_j^1$ is divergent, while $\sum u_j^2$ is convergent. 

By writing 
$$\rho^\beta\left([0,t]_\le^2\right)=\sum_{j\ge1}\E\left[\rho^\beta\left(A_t^{(j)}\right)\right]+\sum_{j\ge1}\left\{\rho^\beta\left(A_t^{(j)}\right)\right\},$$
the first term is infinite and the second has a second moment and so is $a.s.$ finite. This automatically implies that $\rho^\beta([0,t]_\le^2)=\infty$ $a.s.$. The previous equality incites us to consider the variable :

$$\gamma_t^\beta:=\sum_{j\ge1}\left\{\rho^\beta\left(A_t^{(j)}\right)\right\}\quad\left(=\int_{[0,t]_\le^2}\left\{\delta_{\{0\}}\left(U_{s_2}^\beta-U_{s_1}^\beta\right)\right\}\,\d s_1\,\d s_2\right)$$
which converges in $L^2$. The occupation density formula (\ref{equation:occupation density formula for self-intetrsection of a stable processes}) becomes 
$$\int_{A_\le}\left(f\left(U_{s_2}^\beta-U_{s_1}^\beta\right)-\E\left[f\left(U_{s_2}^\beta-U_{s_1}^\beta\right)\right]\right)\,\d s_1\,\d s_2=\int_{\R^d}f(x)\gamma^\beta(x,A_\le)\,\d x$$
for Borel sets $A\subseteq\R_+^2$ and Borel functions $f:\R^d\rightarrow\R_+$, and this time the mapping $x\mapsto\gamma^\beta(x,\cdot)$ is vaguely continuous on $\R^d$ (see \cite{Rosen1996}). Furthermore, the scaling property (\ref{equation:scaling of intersection local time of 2 stable processes}) transfers over to $\gamma^\beta$ :
$$\gamma_{kt}^\beta\overset{(\d)}=k^{2-d/\beta}\gamma_t^\beta\quad\text{for all $k,t>0$,}$$
which may also be seen using the occupation density formula. Also for all $0\le s<t$, by introducing the notation
$$\gamma_{s,t}^\beta:=\gamma^\beta([0,s]\times[s,t]),$$
we see that for all $s,t\ge0$, by writing
$$[0,t+s]_\le^2=[0,t]_\le^2\cup[t,t+s]_\le^2\cup[0,t]\times[t,t+s]$$
and using the independence and stationarity of the increments of $U^\beta$, the following equality in distribution holds :
\begin{equation}
    \label{equation:decomposition of self-intersection local time}
    \gamma_{t+s}^\beta\overset{\mathrm{(d)}}=\gamma_t^\beta+\tilde\gamma_s^\beta+\gamma_{t,t+s}^\beta,
\end{equation}
where $\left(\tilde\gamma^\beta_s\right)_{s\ge0}$ is the renormalized self-intersection local time of $\left(\tilde U^\beta_s:=U^\beta_{t+s}-U^\beta_t\right)_{s\ge0}$. Note that the three variables appearing in the right-hand side of (\ref{equation:decomposition of self-intersection local time}) are independent.

Another construction of the renormalized self-intersection is possible. Indeed, one can start from the clear formal definition
\begin{equation}
\label{def:formal definition of renormalized SILT}
    \gamma_t^\beta=\int_{[0,t]_\le^2}\delta_{\{0\}}\left(U_{s_2}^\beta-U_{s_1}^\beta\right)\,\d s_1\,\d s_2-\E\left[\int_{[0,t]_\le^2}\delta_{\{0\}}\left(U_{s_2}^\beta-U_{s_1}^\beta\right)\,\d s_1\,\d s_2\right],
\end{equation}
define an alternate family of continuous and bounded functionals $\left(\gamma_{\varepsilon,t}^\beta\right)_{\varepsilon,t>0}$ by replacing $\delta_{\{0\}}$ by an approximation $p_\varepsilon$ (like the heat-kernel for example) in (\ref{def:formal definition of renormalized SILT}), and showing some uniform Hölder continuity in $\varepsilon$ using Kolmogorov's criterion (see \cite{LDPforRSILTBassChenRosen2005}), then letting $\varepsilon\rightarrow0$ and showing existence of the limit in some appropriate sense. The reason for recalling the (simpler) construction of $\gamma^\beta$ using dyadic decompositions is that the construction shall be exploited in Section \ref{section:the case 1<=d/beta<3/2}.

\section{The FCLT for the Range}
\label{section:FCLT}

\subsection{The case $1\le d/\beta<3/2$}
\label{section:the case 1<=d/beta<3/2}

Let $T>0$. For all $n\ge1$ and $t\in[0,T]$, define :
$$\overline R_t^{(n)}:=\frac{h(n)^2b_\beta(n)^d}{n^2}\left\{R_{\lfloor nt\rfloor}\right\},\quad\overline{\mathcal R}_t^{(n)}:=\frac{h(n)^2b_\beta(n)^d}{n^2}\left\{\mathcal R_{nt}\right\}.$$

As explained in the strategy of proof, we shall first show tightness of the sequence of laws of $\overline{\mathcal R}^{(n)}$ using Kolmogorov's criterion, then show the convergence of the finite-dimensional distributions of $\overline R^{(n)}$ using Cramér-Wold. In the case where $X$ is transient, $h(n)$ converges to $q^{-1}$ as $n\rightarrow\infty$, where $q=\p(\forall k\ge1,X_k\neq 0)$. When $X$ is transient, then $h(n)$ goes to infinity as $n\rightarrow\infty$. In either case, it is well known that $h$ is slowly varying. Throughout this section, we shall leave $h$ regardless, but one can keep in mind that in the transient case, the results obtained concern the quantity $(b_\beta(n)^d/n^2)\left\{R_{\lfloor nt\rfloor}\right\}$ and the limiting variables are simply to be multiplied by $q^2$. In this setting, Theorem $6.8.$ from \cite{LeGallRosenRangeOfStableRandomWalks1991} yields :
\begin{equation}
\label{convergence:CLT for 1<=d/beta<3/2}
            \overline R_1^{(n)}\distribution-\gamma_1^\beta.
\end{equation}

\textit{Step 1:} We see that the following generalization of \cite{LeGallRosenRangeOfStableRandomWalks1991}, Lemma $6.7.$, holds for all $p\ge1$ :
\begin{equation}
    \label{ineq:bound 2p-th centered moment}
    \E\left[\{R_n\}^{2p}\right]\le C\frac{n^{2p(2-d/\beta)}}{s_\beta(n)^{2dp}h(n)^{4p}} \le C_m\frac{n^{2p(2-d/\beta)}}{s_\beta(n+m)^{2dp}h(n+m)^{4p}}
\end{equation}
where $m\ge 1$ is fixed in such a way that $s_\beta(n+m)h(n+m)\ge C>0$ for all $n\ge1$ and for some constant $C_m<\infty$ depending only on $m$, where we also replaced $b_\beta(n)$ by $n^{1/\beta}s_\beta(n)$. We show the following Lemma, which will automatically yield tightness and uniform local Hölder regularity of the processes $\overline{\mathcal R}^{(n)}$.
\begin{lemma}
\label{lemma:Hölder bound moments of scaled range 1<=d/beta<3/2}
    Let $p\ge1$. Then for any $\eta>0$ sufficiently small, there is some $C_{\eta,p,T}>0$ such that for all $n\ge1$, $(s,t)\in[0,T]_\le^2$,
    \begin{equation}
        \label{ineq:Kolmogorov bound of interpolated rescaled range}
        \E\left[\left(\overline{\mathcal R_t}^{(n)}-\overline{\mathcal R_s}^{(n)}\right)^{2p}\right]\le C_{\eta,p,T}|t-s|^{2p(2-d/\beta)-\eta}.
    \end{equation}
\end{lemma}

\begin{proof}
    Note that 
    \begin{align*}
    \mathcal R_{nt}-\mathcal R_{ns}&=R_{\lfloor nt\rfloor}-R_{\lfloor ns\rfloor}+A_{n,s,t}\\
        &= \left|X(\lfloor ns\rfloor+1,\lfloor nt\rfloor)\right|-|X(0,\lfloor ns\rfloor)\cap X(\lfloor ns\rfloor+1,\lfloor nt\rfloor)|+A_{n,s,t},
    \end{align*}
    where $A_{n,s,t}:=(R_{\lfloor nt\rfloor+1}-R_{\lfloor nt\rfloor})(nt-\lfloor nt\rfloor)-(R_{\lfloor ns\rfloor+1}-R_{\lfloor ns\rfloor})(ns-\lfloor ns\rfloor)$. Since $|A_{n,s,t}|\le2$, we have for all $p\ge1$ and some  $C_p>0$ that may eventually change from line to line, by letting $r_{s,t}^n:=\lfloor nt\rfloor-\lfloor ns\rfloor-1$, that
    \begin{align}
    \label{ineq:Hölder decomposition 1<= d/beta < 3/2}
        \E\left[\left\{\mathcal R_{nt}-\mathcal R_{ns}\right\}^{2p}\right]&\overset{\mathrm{(\ref{ineq:bound for centered 2p-th moment of intersections})},\mathrm{(\ref{ineq:bound for centered moments of intersections})}}\le C_p\left(\E\left[\left\{R_{r_{s,t}^n}\right\}^{2p}\right]+\E\left[I_{\lfloor ns\rfloor,r_{s,t}^n}\right]^{2p}\right)\nonumber\\
        &\overset{\mathrm{(\ref{ineq:bound 2p-th centered moment})}}\le C_p\left(\E\left[\left\{R_{r_{s,t}^n}\right\}^{2p}\right]+\E\left[I_{\lfloor ns\rfloor,\lfloor{n(t-s)}\rfloor}\right]^{2p}\right),
    \end{align}
    since $r_{s,t}^n\le \lfloor n(t-s)\rfloor$. We define $V(x):=s_\beta(x+m)^{2dp}h(x+m)^{4p}$ for $x\ge-m$. Since $s_\beta$ and $h$ are both slowly varying, so is $V$ and we have for all $x\ge0$ and some $C>0$ not depending on $x$ that :
    \begin{equation*}
        V(x-m)\le C V(x).
    \end{equation*}
    As a consequence, we have for all $n\ge1$,
    \begin{equation*}
         h(n)^{4p}s_\beta(n)^{2dp}\le CV(n),
    \end{equation*}
    and so the fact that $r_{s,t}^n\le \lfloor n(t-s)\rfloor$ yields for $\eta$ sufficiently small
    \begin{align}
        \frac{h(n)^{4p}b_\beta(n)^{2dp}}{n^{4p}}\E\left[\left\{R_{r_{s,t}^n}\right\}^{2p}\right]&\overset{\mathrm{(\ref{ineq:bound 2p-th centered moment})}}\le C\left(\frac{r_{s,t}^n}{n}\right)^{2p(2-d/\beta)}\frac{V(n)}{V(r_{s,t}^n)}\nonumber\\
        &\overset{\mathrm{(\ref{ineq:Potter bounds for regularly varying functions})}}\le C_\eta\left(\frac{r_{s,t}^n}{n}\right)^{2p(2-d/\beta)-\eta}\vee\left(\frac{r_{s,t}^n}{n}\right)^{2p(2-d/\beta)+\eta}\nonumber\\
        &\le C_{\eta,T}\left(\frac{ r_{s,t}^n}{n}\right)^{2p(2-d/\beta)-\eta}\nonumber\\
        &\le C_{\eta,T}(t-s)^{2p(2-d/\beta)-\eta}\label{ineq:Holder bound for 2p-th centered moment},
    \end{align}
    where we used in the third inequality the fact that $(r_{s,t}^n/n)^{2\eta}$ is uniformly bounded by some constant $C_T>0$, since $s,t\in[0,T]$. Furthermore, by Lemma \ref{lemma:Hölder bound for moments of I}, we have 
    \begin{equation}
        \frac{h(n)^{4p}b_\beta(n)^{2dp}}{n^{4p}}\E\left[I_{\lfloor ns\rfloor,\lfloor n(t-s)\rfloor}\right]^{2p}\le C_{\eta,p,T}|t-s|^{2p(2-d/\beta)-\eta}.
    \end{equation}
    Combining this inequality and (\ref{ineq:Holder bound for 2p-th centered moment}) with (\ref{ineq:Hölder decomposition 1<= d/beta < 3/2}) yields the desired result
    \end{proof}
    \begin{corollary}
        \label{corollary:local uniform Hölder continuity of scaled range}
    There is a version of $\left(\overline{\mathcal R}^{(n)}\right)_{n\ge1}$ such that for all $T>0$, $\chi\in(0,2-d/\beta)$, there is some $a.s.$ finite random variable $C(T,\chi)$ depending only on $T$ and $\chi$ such that $a.s.$, for all $(s,t)\in[0,T]_\le^2$ and $n\ge1$,
        \begin{equation}
            \label{ineq:local Hölder continuity of scaled range}
            \left|\overline{\mathcal R_t}^{(n)}-\overline{\mathcal R}_s^{(n)}\right|\le C(T,\chi)|t-s|^\chi.
        \end{equation}
    In particular, the sequence $\left(\overline{\mathcal R}^{(n)}\right)_{n\ge1}$ is tight in $\mathcal C([0,T])$.
    \end{corollary}
    \begin{proof}
    This is an immediate consequence of (\ref{ineq:Kolmogorov bound of interpolated rescaled range}) and Kolmogorov's continuity Theorem, since for all $\chi\in(0,2-d/\beta)$, one may choose $p$ sufficienly large and $\eta$ sufficiently small so that $\chi<2-d/\beta-(\eta+1)/2p$.
    \end{proof}

\textit{Step 2:} For sake of simplicity, we only treat the case of $2$-dimensional marginals. By (\ref{convergence:CLT for 1<=d/beta<3/2}), Lemma \ref{lemma:limit of f(x_n)/f(y_n) (regular variation)} and the fact that $h$ is slowly varying, an application of Slutsky's Theorem yields :
$$\overline R_t^{(n)}=\left(\frac{\lfloor nt\rfloor h(n)b_\beta(n)^{d/2}}{nh(\lfloor nt\rfloor)b_\beta(\lfloor nt\rfloor)^{d/2}}\right)^2\frac{h(\lfloor nt\rfloor)^2b_\beta(\lfloor nt\rfloor)^d}{\lfloor nt\rfloor^2}\left\{R_{\lfloor nt\rfloor}\right\}\distribution -t^{2-d/\beta}\gamma_1^\beta\overset{\mathrm{(d)}}=-\gamma_t^\beta$$
by the scaling property of self-intersection local times. For $n,m\ge1$, $j\ge1$, $i\in[\![1,j]\!]$, we introduce the following quantities :
$$R_n^{(i,j)}:=\left|X\left(\frac{(i-1)n}{j},\frac{in}{j}\right)\right|,$$
$$I_{n}^{(i,j)}:=\left|X\left(\frac{(2i-2)n}{j},\frac{(2i-1)n}{j}\right)\cap X\left(\frac{(2i-1)n}{j},\frac{2in}{j}\right)\right|\quad(i\in[\![1,\lfloor j/2\rfloor]\!]),$$
$$I_{n,m}^{(i,j,k)}:=\left|X\left(\frac{(i-1)n}{k},\frac{in}{k}\right)\cap X\left(n+\frac{(j-1)m}{k},n+\frac{jm}{k}\right)\right|\quad(i\in[\![1,\lfloor j/2\rfloor]\!]).$$
For $s,t\ge0$, we introduce the sets :
$$A_t^{(i,j)}:=\left[\frac{2i-2}{j}t,\frac{2i-1}{j}t\right)\times\left(\frac{2i-1}{j}t,\frac{2i}{j}t\right]\quad(i\in[\![1,\lfloor j/2\rfloor]\!]),$$
$$A_{t,s}^{(i,j,k)}:=\left[\frac{i-1}{k}t,\frac{i}{k}t\right)\times\left(t+\frac{j-1}{k}s,t+\frac{j}{k}s\right]\quad(i,j\in[\![1,k]\!]).$$
Let $t\in(0,T)$, $s\in(0,T-t)$ and $\varphi,\psi\in\R$. Then :
\begin{eqnarray*}
    \varphi\overline R_t^{(n)}+\psi\overline R_{t+s}^{(n)}&=&\frac{h(n)^2b_\beta(n)^d}{n^2}\left(\varphi\left\{R_{\lfloor nt\rfloor}\right\}+\psi\left\{R_{\lfloor nt\rfloor}+\tilde R_{\lfloor ns\rfloor}-I_{\lfloor nt\rfloor,\lfloor ns\rfloor}\right\}\right)\\
    &=&\frac{h(n)^2b_\beta(n)^d}{n^2}\left((\varphi+\psi)\left\{R_{\lfloor nt\rfloor}\right\}+\psi\left\{\tilde R_{\lfloor ns\rfloor}\right\}-\psi\left\{I_{\lfloor nt\rfloor,\lfloor ns\rfloor}\right\}\right)\\
    &=&\frac{h(n)^2b_\beta(n)^d}{n^2}\left(\sum_{k=1}^{2^p}\left\{(\varphi+\psi)R_{\lfloor nt\rfloor}^{(k,2^p)}+\psi\tilde R_{\lfloor ns\rfloor}^{(k,2^p)}\right\}\right.\\
    &&\left.-\sum_{\ell=1}^p\sum_{h=1}^{2^{\ell-1}}\left\{(\varphi+\psi)I_{\lfloor nt\rfloor}^{(h,2^\ell)}+\psi \tilde I_{\lfloor ns\rfloor}^{(h,2^\ell)}\right\}-\psi\sum_{i,j=1}^{2^p}\left\{I_{\lfloor nt\rfloor,\lfloor ns\rfloor}^{(i,j,2^p)}\right\}\right)\\
    &=&\frac{h(n)^2b_\beta(n)^d}{n^2}\sum_{k=1}^{2^p}\left\{(\varphi+\psi)R_{\lfloor nt\rfloor}^{(k,2^p)}+\psi\tilde R_{\lfloor ns\rfloor}^{(k,2^p)}\right\}-\left\{\mathcal F_{\varphi,\psi,s,t}^{n,p}\right\},
\end{eqnarray*}
where 
$$\mathcal F_{\varphi,\psi,s,t}^{n,p}:=\frac{h(n)^2b_\beta(n)^d}{n^2}\left(\sum_{\ell=1}^p\sum_{h=1}^{2^{\ell-1}}\left((\varphi+\psi)I_{\lfloor nt\rfloor}^{(h,2^\ell)}+\psi \tilde I_{\lfloor ns\rfloor}^{(h,2^\ell)}\right)+\psi\sum_{i,j=1}^{2^p}I_{\lfloor nt\rfloor,\lfloor ns\rfloor}^{(i,j,2^p)}\right).$$
Firstly, we have
\begin{eqnarray*}
    \E\left[\left(\sum_{k=1}^{2^p}\left\{(\varphi+\psi)R_{\lfloor nt\rfloor}^{(k,2^p)}+\psi\tilde R_{\lfloor ns\rfloor}^{(k,2^p)}\right\}\right)^2\right]&\overset{\mathrm{(\indep)}}{=}&\sum_{k=1}^{2^p}\left((\varphi+\psi)^2\E\left[\left\{R_{\lfloor nt\rfloor}^{(k,2^p)}\right\}^2\right]\right.\\
    &&\left.\quad\quad\quad+\psi^2\E\left[\left\{\tilde R_{\lfloor ns\rfloor}^{(k,2^p)}\right\}^2\right]\right)\\
    &\overset{(\ref{ineq:bound 2p-th centered moment})}\le&C_{\varphi,\psi}\sum_{k=1}^{2^p}\left(\sum_{u=s,t}\frac{\lfloor nu\rfloor^42^{-4p}}{b_\beta(\lfloor nu\rfloor2^{-p})^{2d}h(\lfloor nu\rfloor 2^{-p})^4}\right)\\
    &=&C_{\varphi,\psi}\sum_{u=s,t}\frac{\lfloor nu\rfloor^42^{-3p}}{b_\beta(\lfloor nu\rfloor2^{-p})^{2d}h(\lfloor nu\rfloor 2^{-p})^4}\\
    &=:&C_{\varphi,\psi}\mathcal S_{n,p,s,t},
\end{eqnarray*}
where $C_{\varphi,\psi}>0$ is a constant depending only on $\varphi,\psi$, and where the inequality was obtained from the fact that for all $k\in[\![1,2^p]\!]$ and $u=s,t$, $$R_{\lfloor nu\rfloor}^{(k,2^p)}\overset{\mathrm{(d)}}=R_{\lfloor nu\rfloor2^{-p}}.$$
For $\eta>0$ and $n$ sufficiently large (depending on $\eta$), since $b_\beta$ and $h$ are respectively of regular variation of index $1/\beta$ and slowly varying functions, Lemma \ref{lemma:limit of f(x_n)/f(y_n) (regular variation)} yields :
\begin{eqnarray*}
    \frac{h(n)^2b_\beta(n)^{2d}}{n^4}\mathcal S_{n,p,s,t}&=&\sum_{u=s,t}\left(\frac{\lfloor nu\rfloor}{u}\right)^4\left(\frac{b_\beta(\lfloor nu\rfloor 2^{-p})}{b_\beta(n)}\right)^{2d}\left(\frac{h(\lfloor nu\rfloor2^{-p})}{h(n)}\right)^42^{-3p}\\
    &\le&2C_{s,t,\eta}(1+\eta)^{2d+4}2^{-3p},
\end{eqnarray*}
where $C_{s,t,\eta}:=(s+\eta)^4\vee(t+\eta)^4$, whence :
\begin{equation}
    \label{ineq:bound big range term 1 <= d/beta<3/2}
    \E\left[\left(\frac{h(n)^2b_\beta(n)^d}{n^2}\sum_{k=1}^{2^p}\left\{(\varphi+\psi)R_{\lfloor nt\rfloor}^{(k,2^p)}+\psi\tilde R_{\lfloor ns\rfloor}^{(k,2^p)}\right\}\right)^2\right]\le\tilde C2^{-3p}
\end{equation}
where $\tilde C$ is a constant depending only on $\varphi,\psi,s,t,\eta$. 

Furthermore, with a clear notation, we have 
$$\mathcal F_{\varphi,\psi,s,t}^{n,p}=F\left(\left(I_{\lfloor nt\rfloor}^{(h,2^\ell)}\right)_{\substack{1\le\ell\le p\\1\le h\le 2^{\ell-1}}},\left(\tilde I_{\lfloor ns\rfloor}^{(h,2^\ell)}\right)_{\substack{1\le\ell\le p\\1\le h\le 2^{\ell-1}}},\left(I_{\lfloor nt\rfloor,\lfloor ns\rfloor}^{(i,j,2^p)}\right)_{1\le i,j\le 2^p}\right)$$
for some continuous function $F$. By directly adapting the proof of Theorem $6.6.$ of \cite{LeGallRosenRangeOfStableRandomWalks1991}, it is seen that :
\begin{eqnarray*}
    \left\{\mathcal F_{\varphi,\psi,s,t}^{n,p}\right\}&\distribution&\left\{F\left(\left(\alpha^\beta\left(A_t^{(h,2^\ell)}\right)\right)_{\ell,h},\left(\tilde\alpha^\beta\left(A_s^{(h,2^\ell)}\right)\right)_{\ell,h},\left(\alpha^\beta\left(A_{t,s}^{(i,j,2^p)}\right)\right)_{i,j}\right)\right\}\\
    &=&\sum_{\ell=1}^p\sum_{h=1}^{2^{\ell-1}}\left((\varphi+\psi)\left\{\alpha^\beta\left(A_t^{(h,2^\ell)}\right)\right\}\right.\\
    &&\left.\quad\quad\quad\quad\quad+\psi\left\{\tilde\alpha^\beta\left(A_s^{(h,2^\ell)}\right)\right\}\right)+\psi\sum_{i,j=1}^{2^p}\left\{\alpha^\beta\left(A_{t,s}^{(i,j,2^p)}\right)\right\}\\
    &=:&A_p+B_p+C_p,
\end{eqnarray*}
where the indexes in the first line are such that $1\le\ell\le p$, $1\le h\le 2^{\ell-1}$ and $1\le i,j\le 2^p$. Noting firstly that the sequences $(A_p)_{p\ge1}$, $(B_p)_{p\ge1}$ and $(C_p)_{p\ge1}$ are independent, and that by construction we have :
$$A_p\underset{p\rightarrow\infty}\implies (\varphi+\psi)\gamma_t^\beta,\quad B_p\underset{p\rightarrow\infty}\implies\psi\tilde\gamma_s^\beta,\quad C_p\distribution\psi\gamma^\beta([0,t]\times[t,t+s]),$$
then :
$$A_p+B_p+C_p\underset{p\rightarrow\infty}\implies (\varphi+\psi)\gamma_t^\beta+\psi\tilde \gamma_s^\beta+\psi\gamma^\beta([0,t]\times[t,t+s])\overset{\mathrm{(d)}}=\varphi\gamma_t^\beta+\psi\gamma_{t+s}^\beta,$$
where the last equality in distribution follows from (\ref{equation:decomposition of self-intersection local time}). Therefore, we obtain  :
$$\varphi\overline R_t^{(n)}+\psi\overline R_{t+s}^{(n)}\distribution -\varphi\gamma_t^\beta-\psi\gamma_{t+s}^\beta$$
by using the following fact :

\begin{lemma}
    Let $(X_n)_{n\ge1}$ be real random variables and suppose that there are real random variables $(A_{n,p})_{n,p\ge1}$, $(B_{n,p})_{n,p\ge1}$, $(B_p)_{p\ge1}$ and $X$ such that :
    \begin{itemize}
        \item[(i)] for all $n,p\ge1$, $X_n=A_{n,p}+B_{n,p}$ $a.s.$, 
        \item[(ii)] $\underset{p\rightarrow\infty}\lim\underset{n\rightarrow\infty}\limsup\lVert A_{n,p}\rVert_2=0$,
        \item[(iii)] for all $p\ge1$, $B_{n,p}\distribution B_p$ and $B_p\underset{p\rightarrow\infty}\implies X$.
    \end{itemize}
    Then we have $X_n\distribution X$.
\end{lemma}

\begin{proof}
    Let $f$ be a bounded $L$-Lipschitz test function for some $L>0$. Then for all $n,p\ge1$ :
    \begin{eqnarray*}
        \left|\E[f(X_n)-f(X)]\right|&\le&\left|\E[f(X_n)-f(B_{n,p})]\right|+\left|\E[f(B_{n,p})-f(B_p)]\right|+\left|\E[f(B_p)-f(X)]\right|
    \end{eqnarray*}
    The first term is bounded by $L\lVert A_{n,p}\rVert_2$ by Jensen's inequality and the second vanishes as $n\rightarrow\infty$ by $(iii)$. Whence for all $p\ge1$ :
    $$\underset{n\rightarrow\infty}\limsup\left|\E[f(X_n)-f(X)]\right|\le L\text{ }\underset{n\rightarrow\infty}\limsup\lVert A_{n,p}\rVert_2+|\E[f(B_p)-f(X)]|,$$
    and we conclude by letting $p\rightarrow\infty$.
\end{proof}

This concludes the proof of (\ref{convergence:FCLT for 1<=d/beta<3/2}). Combining the results of \textit{Step $2.$} and Corollary \ref{corollary:local uniform Hölder continuity of scaled range} immediately yields the

\begin{theorem}
    The mapping $t\mapsto\gamma_t^\beta$ has a version that is $a.s.$ locally $\chi$-Hölder continuous for all $\chi\in(0,2-d/\beta)$.
\end{theorem}

\begin{remark}
    It is likely that the local Hölder exponent we get for $\gamma^\beta$ is optimal. Indeed, the laws of iterated logarithms obtained in \cite{LDPforRSILTBassChenRosen2005} show that the sample paths of $t\mapsto\gamma_t^\beta$ in the case where $U^\beta$ is isotropic can't be $(2-d/\beta)$-Hölder, and a similar result should hold in the anisotropic case.
\end{remark}

\subsection{The case $d/\beta\ge3/2$}
\label{section:the case d/beta=3/2}

Let $T>0$. For all $n\ge1$ and $t\in[0,T]$, define :
\begin{equation*}
    \overline R_t^{(n)}:=\frac{1}{\sqrt{ng(n)}}\left\{R_{\lfloor nt\rfloor}\right\},\quad \overline{\mathcal R}_t^{(n)}:=\frac{1}{\sqrt{ng(n)}}\left\{\mathcal R_{nt}\right\}.
\end{equation*}
By Theorem $4.7.$ of \cite{LeGallRosenRangeOfStableRandomWalks1991}
\begin{equation}
    \label{convergence:CLT for d/beta=3/2}
    \overline R_1^{(n)}\distribution\sigma N
\end{equation}
for some constant $\sigma>0$ which depends only on $X$ and can be made explicit.

\textit{Step 1:} We start of by noticing, by examining the given proof, that the inequality of $p.667$ of \cite{LeGallRosenRangeOfStableRandomWalks1991} has the following easy generalization
\begin{equation}
    \label{ineq:bound for centered 2p-th moment of range d/beta=3/2}
    \E\left[\left\{R_n\right\}^{2p}\right]\le Cn^pg(n)^p, \quad n,p\ge1.
\end{equation}
We show a similar Lemma to that of the previous section.
\begin{lemma}
    Let $p\ge1$. Then for any $\eta>0$ sufficiently small, there is some $C_{\eta,p,T}>0$ such that for all $n\ge1$, $(s,t)\in[0,T]_\le^2$,
    \begin{equation}
        \E\left[\left(\overline{\mathcal R}_t^{(n)}-\overline{\mathcal R}_s^{(n)}\right)^{2p}\right]\le C_{\eta,p,T}|t-s|^{p-\eta}
    \end{equation}
\end{lemma}
\begin{proof}
    As in the previous section, we have
    \begin{align}
    \label{ineq:Hölder decomposition d/beta=3/2}
        \E\left[\left\{\mathcal R_{nt}-\mathcal R_{ns}\right\}^{2p}\right]&\overset{\mathrm{(\ref{ineq:bound for centered 2p-th moment of intersections})},\mathrm{(\ref{ineq:bound for centered moments of intersections})}}\le C_p\left(\E\left[\left\{R_{r_{s,t}^n}\right\}^{2p}\right]+\E\left[I_{\lfloor ns\rfloor,r_{s,t}^n}\right]^{2p}\right)\nonumber\\
        &\overset{\mathrm{(\ref{ineq:bound for centered 2p-th moment of range d/beta=3/2})}}\le C_p\left(\left(r_{s,t}^n\right)^pg\left(r_{s,t}^n\right)^p+\E\left[I_{\lfloor ns\rfloor,\lfloor{n(t-s)}\rfloor}\right]^{2p}\right)\nonumber\\
        &\le C_p\left(\lfloor n(t-s)\rfloor^pg\left(\lfloor n(t-s)\rfloor\right)^p+\E\left[I_{\lfloor nT\rfloor,\lfloor{n(t-s)}\rfloor}\right]^{2p}\right)
    \end{align}
    since $r_{s,t}^n\le\lfloor n(t-s)\rfloor$ and the mappings $n\mapsto ng(n)$ and $(n,m)\mapsto I_{n,m}$ are increasing. Reasoning as in the beginning of this section, $g$ is easily seen to be slowly varying. Consequently, for $\eta>0$ small,
    \begin{equation}
    \label{ineq:Hölder bound for first term in d/beta=3/2}
        \frac{\lfloor n(t-s)\rfloor^pg(\lfloor n(t-s)\rfloor)^p}{n^pg(n)^p}\le C_{\eta}|t-s|^{p-\eta}
    \end{equation}
    using the Potter bounds. Next, Lemma \ref{lemma:Hölder bound for moments of I} yields 
    \begin{align}
    \label{ineq:Hölder bound for second term in d/beta=3/2}
        \frac{1}{n^pg(n)^p}\E\left[I_{\lfloor nT\rfloor,\lfloor n(t-s)\rfloor}\right]^{2p}\le C_{\eta,T}|t-s|^{p-\eta}.
    \end{align}
    Plugging (\ref{ineq:Hölder bound for first term in d/beta=3/2}) and (\ref{ineq:Hölder bound for second term in d/beta=3/2}) into (\ref{ineq:Hölder decomposition d/beta=3/2}) and multiplying by $n^{-p}g(n)^{-p}$ yields the desired result.
\end{proof}
As in the previous section, we immediately deduce the following Corollary
\begin{corollary}
        \label{corollary:local uniform Hölder continuity of scaled range d/beta=3/2}
    There is a version of $\left(\overline{\mathcal R}^{(n)}\right)_{n\ge1}$ such that for all $T>0$, $\chi\in(0,1/2)$, there is some $a.s.$ finite random variable $C(T,\chi)$ depending only on $T$ and $\chi$ such that $a.s.$, for all $(s,t)\in[0,T]_\le^2$ and $n\ge1$,
        \begin{equation}
            \label{ineq:local Hölder continuity of scaled range d/beta=3/2}
            \left|\overline{\mathcal R_t}^{(n)}-\overline{\mathcal R}_s^{(n)}\right|\le C(T,\chi)|t-s|^\chi.
        \end{equation}
    In particular, the sequence $\left(\overline{\mathcal R}^{(n)}\right)_{n\ge1}$ is tight in $\mathcal C([0,T])$.
    \end{corollary}

\textit{Step 1':} We focus here on proving tightness in the regime $d/\beta>3/2$, since convergence of finite-dimensional distributions was already proved in \cite{CyganSandricSebekFCLTCapacityAndRange2019}. In this case, $\lVert g\rVert_\infty<\infty$ and so we neglect $g$ in our calculations. We do this so that our application may cover the regime $d/\beta>3/2$ as well. Starting from (\ref{ineq:Hölder decomposition d/beta=3/2}), we see that the first term is treated identically, and so we study the second. We may apply Lemma \ref{lemma:Hölder bound for moments of I} to show that it is $O(n^p|t-s|^{p-\eta})$ if $d/\beta\in(3/2,2)$. By Corollary $3.2.$ of \cite{LeGallRosenRangeOfStableRandomWalks1991}, we have $\sup_n\E[I_{n,n}]<\infty$ if $d/\beta>2$, and so in this case we may neglect the second term altogether up to increasing $C_p$. It remains to treat the case $d/\beta=2$. Once again, the first term in (\ref{ineq:Hölder decomposition d/beta=3/2}) is handled exactly as before, and for the second we bound it by $l(\lfloor nT\rfloor)$ where $l$ is some slowly varying function, which is possible thanks to Corollary $3.2.$ of \cite{LeGallRosenRangeOfStableRandomWalks1991}. Therefore, once the expression is multiplied by $n^{-p}$, the second term is bounded by $C_{p,\eta}(n^{\eta-p}+n^{-\eta-p})$ by the Potter bounds. As in the proof of Lemma \ref{lemma:Hölder bound for moments of I}, if $t-s<1/n$ then there is nothing to prove, and so we suppose that $t-s\ge1/n$, in which case the second term is bounded by $C_{\eta,p}(|t-s|^{p-\eta}+|t-s|^{p+\eta})\le C_{\eta,p,T}|t-s|^{p-\eta}$. All in all, we have shown that in this regime,
\begin{equation*}
    \E\left[\left\{\frac{\mathcal R_{nt}-\mathcal R_{ns}}{\sqrt{n}}\right\}^{2p}\right]\le C_{\eta,p,T}|t-s|^{p-\eta},
\end{equation*}
and so Corollary \ref{corollary:local uniform Hölder continuity of scaled range d/beta=3/2} remains true if $d/\beta>3/2$.
    
\textit{Step 2:} For sake of simplicity we only write the convergence of $2$-dimensional marginals, but the proof is identical for marginals of higher dimension. Let $s,t\in[0,T]$ and $\varphi,\psi\in\R$. We wish to show that 
$$\varphi\overline R_s^{(n)}+\psi\overline R_t^{(n)}\distribution\varphi W_{\sigma^2s}+\psi W_{\sigma^2t}.$$
By writing
$$R_{\lfloor nt\rfloor}=R_{\lfloor ns\rfloor}+\tilde R_{\lfloor nt\rfloor-\lfloor ns\rfloor}-I_{\lfloor ns\rfloor,\lfloor nt\rfloor-\lfloor ns\rfloor},$$
we obtain 
\begin{equation}
    \label{eq:decomposition for the 2d marginals d/beta=3/2}
    a\overline R_s^{(n)}+\psi\overline R_t^{(n)}=(\varphi+\psi)\frac{\left\{R_{\lfloor ns\rfloor}\right\}}{\sqrt{ng(n)}}+\psi\frac{\left\{\tilde R_{\lfloor nt\rfloor-\lfloor ns\rfloor}\right\}}{\sqrt{ng(n)}}-\psi\frac{\left\{I_{\lfloor ns\rfloor,\lfloor nt\rfloor-\lfloor ns\rfloor}\right\}}{\sqrt{ng(n)}}.
\end{equation}
From (\ref{convergence:CLT for d/beta=3/2}), Lemma \ref{lemma:limit of f(x_n)/f(y_n) (regular variation)} and Slutsky's Theorem, we see that the first two terms in the right hand side of (\ref{eq:decomposition for the 2d marginals d/beta=3/2}) respectively converge in distribution to $(a+\psi)s^{1/2}\sigma N$ and $\psi(t-s)^{1/2}\sigma\tilde N$ where $N,\tilde N$ are $i.i.d.$ standard normal random variables, whereas the third term converges in probability to $0$ by Markov's inequality and the estimates of \textit{Step 1}. Therefore,
$$\varphi\overline R_s^{(n)}+\psi\overline R_t^{(n)}\distribution(\varphi+\psi)s^{1/2}\sigma N+\psi(t-s)^{1/2}\sigma\tilde N\overset{\mathrm{(d)}}=\varphi W_{\sigma^2s}+\psi W_{\sigma^2 t}$$
which concludes the proof of (\ref{convergence:FCLT for d/beta=3/2}).

\subsection{The case $d/\beta<1$}

Let $T>0$. For all $n\ge1$ and $t\in[0,T]$, define 
\begin{equation*}
    \overline R_t^{(n)}:=b_\beta(n)^{-1}R_{\lfloor nt\rfloor},\quad\overline{\mathcal R}_t^{(n)}:=b_\beta(n)^{-1}\mathcal R_{nt}.
\end{equation*}
Then by Theorem $7.1.$ and $(7.a)$ of \cite{LeGallRosenRangeOfStableRandomWalks1991}, we have 
\begin{equation}
    \label{convergence:LT for d/beta<1}
    \overline R_1^{(n)}\distribution\lambda\left(U^\beta(0,1)\right)\quad\text{and}\quad \underset{n\rightarrow\infty}\lim \E\left[\overline R_1^{(n)}\right]=\E\left[\lambda\left(U^\beta(0,1)\right)\right]
\end{equation}
respectively.

\textit{Step 1:} As in the previous sections, we state the
\begin{lemma}
    Let $p\ge1$. Then for any $\eta>0$ sufficiently small, there is some $C_{\eta,p,T}>0$ such that for all $n\ge1$, $(s,t)\in[0,T]_\le^2$,
    \begin{equation}
        \E\left[\left(\overline{\mathcal R}_t^{(n)}-\overline{\mathcal R}_s^{(n)}\right)^{2p}\right]\le C_{\eta,p,T}|t-s|^{2p/\beta-\eta}
    \end{equation}
\end{lemma}
\begin{proof}
    Reasoning as we already have, we get 
\begin{equation}
\label{ineq:Hölder decomposition d/beta<1}
    \E\left[\left(\overline{\mathcal R}_t^{(n)}-\overline{\mathcal R}_s^{(n)}\right)^{2p}\right]\le C_p\left(b_\beta(n)^{-2p}\E\left[R_{\lfloor n(t-s)\rfloor}^{2p}\right]+\left(b_\beta(n)^{-1}\E\left[I_{\lfloor nT\rfloor,\lfloor n(t-s)\rfloor}\right]\right)^{2p}\right).
\end{equation}
The second term is handled by Lemma \ref{lemma:Hölder bound for moments of I} which reads for small $\eta$,
\begin{align}
\label{ineq:Hölder bound second term d/beta<1}
    b_\beta(n)^{-1}\E\left[I_{\lfloor nT\rfloor,\lfloor n(t-s)\rfloor}\right]\le C_{\eta,T}|t-s|^{1/\beta-\eta}.
\end{align}
We now turn our attention to the first term of (\ref{ineq:Hölder decomposition d/beta<1}). It is easily seen using the Markov property that we have for all $n\ge1$
\begin{equation*}
    \E\left[R_n^{2p}\right]\le C_p\E[R_n]^{2p}.
\end{equation*}
This and (\ref{convergence:LT for d/beta<1}) imply that 
\begin{align}
\label{ineq:Hölder bound first term d/beta<1}
    b_\beta(n)^{-2p}\E\left[R_{\lfloor n(t-s)\rfloor}^{2p}\right]\le C_p\left(\frac{b_\beta(\lfloor n(t-s)\rfloor)}{b_\beta(n)}\right)\le C_{p,\eta}|t-s|^{1/\beta-\eta}
\end{align}
by the Potter bounds. Putting (\ref{ineq:Hölder bound second term d/beta<1}) and (\ref{ineq:Hölder bound first term d/beta<1}) into (\ref{ineq:Hölder decomposition d/beta<1}) concludes the proof
\end{proof}
Again, we get the
\begin{corollary}
    \label{corollary:local uniform Hölder continuity of scaled range d/beta<1}
    There is a version of $\left(\overline{\mathcal R}^{(n)}\right)_{n\ge1}$ such that for all $T>0$, $\chi\in(0,1/\beta)$, there is some $a.s.$ finite random variable $C(T,\chi)$ depending only on $T$ and $\chi$ such that $a.s.$, for all $(s,t)\in[0,T]_\le^2$ and $n\ge1$,
        \begin{equation}
            \label{ineq:local Hölder continuity of scaled range d/beta=3/2}
            \left|\overline{\mathcal R_t}^{(n)}-\overline{\mathcal R}_s^{(n)}\right|\le C(T,\chi)|t-s|^\chi.
        \end{equation}
    In particular, the sequence $\left(\overline{\mathcal R}^{(n)}\right)_{n\ge1}$ is tight in $\mathcal C([0,T])$.
\end{corollary}

\textit{Step 2:}
Let $s,t\in[0,T]$, $s<t$ and $\varphi,\psi\in\R$. The following proof is (up to details) the same as that of Theorem $7.1.$ in \cite{LeGallRosenRangeOfStableRandomWalks1991}, but we provide it for sake of completeness. By Skorokhod's extension of Donsker's Theorem (\cite{SkorokhodLimitTheoremsStochProc1956}), by writing $X_t^{(n)}:=b_\beta(n)^{-1}X_{\lfloor nt\rfloor}$, then $X^{(n)}$ converges in distribution to $U^\beta$ in the $J_1$ topology. By Skorokhod's representation Theorem, we may construct for all $n\ge1$ a process $\mathcal X^{(n)}$ distributed as $X^{(n)}$ and $\mathcal U^\beta$ distributed as $U^\beta$ such that $\mathcal X^{(n)}$ converges $a.s.$ to $\mathcal U^\beta$ in $\mathcal D_T(\R)$. For $r\in[0,T]$, $\varepsilon>0$, we define 
$$\mathcal W_r^\varepsilon:=\left\{x\in\R,d\left(x,\mathcal U^\beta(0,r)\right)\le\varepsilon\right\}.$$
By monotone convergence, we have 
$$a.s.\text{-}\underset{\varepsilon\rightarrow0}\lim\lambda\left(\mathcal W^\varepsilon_r\right)=\lambda\left(\overline{\mathcal U^\beta(0,r)}\right)\overset{a.s.}=\lambda\left(\mathcal U^\beta(0,r)\right),$$
where the second equality follows from the fact that $\mathcal U^\beta$ has a countable number of discontinuities. By construction, for all $\omega\in\Omega$ and $\varepsilon>0$, there exists $n_\varepsilon(\omega)\ge1$ such that for all $r\in[0,T]$, $n\ge n_\varepsilon(\omega)$, $\mathcal X_r^{(n)}\in\mathcal W^\varepsilon_r$ and $b_\beta(n)^{-1}<\varepsilon$. By letting $\tilde R_n:=|\mathcal X^{(n)}(0,n)|$ for $n\ge1$ in such a way that $\tilde R_n\overset{\mathrm{(d)}}=R_n$, we get for all sufficiently large $n$ and all $r\in[0,T]$ :
\begin{eqnarray*}
    b_\beta(n)^{-1}\tilde R_{\lfloor nr\rfloor}&=&b_\beta(n)^{-1}\sum_{y\in\Z}\1_{\{\lfloor y\rfloor/b_\beta(\lfloor nr\rfloor)\in\mathcal X^{(\lfloor nr\rfloor)}(0,\lfloor nr\rfloor)\}}\\
    &=&b_\beta^{-1}(n)\int_\R\1_{\{\lfloor y\rfloor/b_\beta(\lfloor nr\rfloor)\in\mathcal X^{(\lfloor nr\rfloor)}(0,\lfloor nr\rfloor)\}}\,\d y\\
    &=&\frac{b_\beta(\lfloor nr\rfloor}{b_\beta(\lfloor n\rfloor)}\int_\R\1_{\{\lfloor y b_\beta(\lfloor nr\rfloor)\rfloor/b_\beta(\lfloor nr\rfloor)\in\mathcal X^{(\lfloor nr\rfloor)}(0,\lfloor nr\rfloor)\}}\,\d y\\
    &\le&\frac{b_\beta(\lfloor nr\rfloor}{b_\beta(\lfloor n\rfloor)}\int_\R\1_{\{y\in\mathcal W_1^{2\varepsilon}\}}\,\d y=\frac{b_\beta(\lfloor nr\rfloor}{b_\beta(\lfloor n\rfloor)}\lambda\left(\mathcal W_1^{2\varepsilon}\right),
\end{eqnarray*}
whence 
\begin{equation*}
    \underset{n\rightarrow\infty}\limsup \text{ }b_\beta(n)^{-1}\tilde R_{\lfloor nr\rfloor}\le r^{1/\beta}\lambda\left(\mathcal U^\beta(0,1)\right).
\end{equation*}
Therefore, we have 
$$\underset{n\rightarrow\infty}\limsup\text{ }b_\beta(n)^{-1}\left(\varphi\tilde R_{\lfloor ns\rfloor}+\psi\tilde R_{\lfloor nt\rfloor}\right)\le(\varphi s^{1/\beta}+\psi t^{1/\beta})\lambda\left(\mathcal U^\beta(0,1)\right)\quad a.s.,$$
which implies 
\begin{equation*}
    \underset{n\rightarrow\infty}\lim\left(b_\beta(n)^{-1}\left(\varphi\tilde R_{\lfloor ns\rfloor}+\psi\tilde R_{\lfloor nt\rfloor}\right)-(\varphi s^{1/\beta}+\psi t^{1/\beta})\lambda\left(\mathcal U^\beta(0,1)\right)\right)_+=0\quad a.s.
\end{equation*}
As $b_\beta(n)^{-1}\tilde R_n$ is bounded in $L^2$ (see \cite{LeGallRosenRangeOfStableRandomWalks1991}, p.$703$), so is $b_\beta(n)^{-1}\left(\varphi\tilde R_{\lfloor ns\rfloor}+\psi\tilde R_{\lfloor nt\rfloor}\right)$ and thus by uniform integrability 
$$\underset{n\rightarrow\infty}\lim\E\left[\left(b_\beta(n)^{-1}\left(\varphi\tilde R_{\lfloor ns\rfloor}+\psi\tilde R_{\lfloor nt\rfloor}\right)-(\varphi s^{1/\beta}+\psi t^{1/\beta})\lambda\left(\mathcal U^\beta(0,1)\right)\right)_+\right]=0.$$
(\ref{convergence:LT for d/beta<1}) and the regularly varying nature of $b_\beta$ yield 
\begin{equation*}
    \underset{n\rightarrow\infty}\lim\E\left[b_\beta(n)^{-1}\left(\varphi\tilde R_{\lfloor ns\rfloor}+\psi\tilde R_{\lfloor nt\rfloor}\right)-(\varphi s^{1/\beta}+\psi t^{1/\beta})\lambda\left(\mathcal U^\beta(0,1)\right)\right]=0,
\end{equation*}
hence 
\begin{equation*}
    \underset{n\rightarrow\infty}\lim\E\left[\left|b_\beta(n)^{-1}\left(\varphi\tilde R_{\lfloor ns\rfloor}+\psi\tilde R_{\lfloor nt\rfloor}\right)-(\varphi s^{1/\beta}+\psi t^{1/\beta})\lambda\left(\mathcal U^\beta(0,1)\right)\right|\right]=0,
\end{equation*}
and therefore 
\begin{eqnarray*}
    \varphi\overline{R}_s^{(n)}+\psi\overline{R}_t^{(n)}&\overset{\mathrm{(d)}}=&b_\beta(n)^{-1}\left(\varphi\tilde R_{\lfloor ns\rfloor}+\psi\tilde R_{\lfloor nt\rfloor}\right)\\
    &\distribution&(\varphi s^{1/\beta}+\psi t^{1/\beta})\lambda\left(\mathcal U^\beta(0,1)\right)\overset{\mathrm{(d)}}=\varphi\lambda(U^\beta(0,s))+\psi\lambda(U^\beta(0,t))
\end{eqnarray*}
which concludes the proof.

\section{Limit Theorems for the energy functional}

\label{section:LT energy}

\subsection{Proof of the CLT}

The aim of this section is to provide a proof of Theorem \ref{theorem:main result 2}. For sake of brevity, since the proofs are almost identical in each case, we shall only treat the case $1\le d/\beta<3/2$. As mentioned previously, we may work with $\mathcal E$ instead of $E$ and the result shall follow. Let $m:\R_+\rightarrow\R_+$ be a continuously differentiable function of regular variation of index $\chi>d/\beta-2$ with monotone derivative. We let $m_n(x):=m(nx)/m(n)$ for $x\ge0$, $n\ge1$. Then for $t>0$,
\begin{align}
    \frac{h(n)^2b_\beta(n)^d}{m(n)n^2}\{\mathcal E_{nt}\}&=\int_0^{nt}\frac{m(nt-s)}{m(n)}\,\d\overline{\mathcal R}_{s/n}^{(n)}\nonumber\\
    &=\int_0^tm_n(t-s)\,\d\overline{\mathcal R}_s^{(n)}\nonumber\\
    &=-\int_0^tm_n(s)\,\d\overline{\mathcal R}_{t-s}^{(n)}\nonumber\\
    &=-\int_0^tm_n(s)\,\d\left(\overline{\mathcal R}_{t}^{(n)}-\overline{\mathcal R}_{t-s}^{(n)}\right)\nonumber\\
    &=m_n(t)\overline{\mathcal R}_t^{(n)}-\int_0^t\left(\overline{\mathcal R}_{t}^{(n)}-\overline{\mathcal R}_{t-s}^{(n)}\right)\partial_s m_n(s)\,\d s,
\end{align}
where we succesively used a change of variable and integration by parts for Stieltjes integrals. If $(Y_t)_{t\ge0}$ is a stochastic process with continuous sample paths, define for all $0\le s\le t$ the process
\begin{equation*}
    Y_{s,t}:=Y_t-Y_{t-s}.
\end{equation*}
We now invoke Skorokhod's representation Theorem to construct a process $\gamma'^\beta$ distributed as $\gamma^\beta$ and for each $n\ge1$ a process $\mathcal R'$ distributed as $\mathcal R$ each on some probability space such that the following $a.s.$ convergence holds :
\begin{equation}
\label{convergence:a.s. convergence of sup norm of range + SILT to 0}
    \underset{s\in[0,t]}\sup\left|\overline{\mathcal{R}'}_{s,t}^{(n)}+\gamma'^\beta_{s,t}\right|\overset{a.s.}{\underset{n\rightarrow\infty}\longrightarrow}0.
\end{equation}
This is possible thanks to Theorem \ref{theorem:main result}. By the assumptions we made on $m$, we have for all $s\in(0,t]$,
\begin{equation*}
    \underset{n\rightarrow\infty}\lim\partial_sm_n(s)=\chi s^{\chi-1}
\end{equation*}
(see \cite{BinghamGoldieTeugelsRegularVariation1987}, p.$39$). Firstly, \begin{equation}
\label{convergence:a.s. convergence of the first term}
    \left|m_n(t)\overline{\mathcal R'}_t^{(n)}+t^\chi\overline\gamma_t^\beta\right|\underset{n\rightarrow\infty}{\overset{a.s.}\longrightarrow}0.
\end{equation}
Secondly, by using the inequality $|xx'-yy'|\le|x(x'-y')|+|y'(x-y)|$, we get
\begin{align}
    \left|\int_0^t\overline{\mathcal R'}_{s,t}^{(n)}\partial_sm_n(s)\,\d s+\chi\int_0^t\gamma_{s,t}'^\beta s^{\chi-1}\,\d s\right|&\le\int_0^t\left|\overline{\mathcal R'}_{s,t}^{(n)}\right|\left|\partial_sm_n(s)-\chi s^{\chi-1}\right|\,\d s\nonumber\\
    &\quad\quad+\int_0^t\chi s^{\chi-1}\left|\overline{\mathcal R'}_{s,t}^{(n)}+\gamma_{s,t}'^\beta\right|\,\d s\label{ineq:bound for convergence of integral in CLT of energy}.
\end{align}
Fix $\eta\in(0\vee(-\chi),2-d/\beta)$. Since the uniform in $n$ local $\eta$-Hölder continuity of $\overline{\mathcal R}^{(n)}$ and $\gamma^\beta$ transfers over to $\overline{\mathcal R'}^{(n)}$ and $\gamma'^\beta$, then as in the beginning of the proof, the two integrals appearing in the right-hand side of (\ref{ineq:bound for convergence of integral in CLT of energy}) are well defined, and an  application of the dominated convergence Theorem and (\ref{convergence:a.s. convergence of sup norm of range + SILT to 0}) shows that the second one converges to $0$ $a.s.$ as $n\rightarrow\infty$. Furthermore, since $m'$ is regularly varying of index $\chi-1$, then by Potter's bounds, for all $\varepsilon>0$, there is some $C_\varepsilon>0$ such that for all $s\in(0,t]$,
\begin{equation}
\label{ineq:bound for m_n'}
    |\partial_sm_n(s)|=\left|\frac{nm'(n)}{m(n)}\right|\left|\frac{m'(ns)}{m'(n)}\right|\le C_\varepsilon\left(s^{\chi+\varepsilon-1}+s^{\chi-\varepsilon-1}\right),
\end{equation}
since $nm'(n)/m(n)$ converges to $\chi$ as $n\rightarrow\infty$. Therefore, by taking $\varepsilon\in(0,\eta+\chi)$,  we have 
\begin{equation*}
    \left|\overline{\mathcal R'}_{s,t}^{(n)}\right|\left|\partial_sm_n(s)-\chi s^{\chi-1}\right|\le C(t,\eta)\left(C_\varepsilon\right(s^{\eta+\chi+\varepsilon-1}+s^{\eta+\chi-\varepsilon-1})+\chi s^{\eta+\chi-1}),
\end{equation*}
which is integrable on $[0,t]$. Whence by the dominated convergence Theorem, we have 
\begin{equation}
\label{convergence:a.s. convergence of the integral}
    \left|\int_0^t\overline{\mathcal R'}_{s,t}^{(n)}\partial_sm_n(s)\,\d s+\chi\int_0^t\overline\gamma_{s,t}^\beta s^{\chi-1}\,\d s\right|\overset{a.s.}{\underset{n\rightarrow\infty}\longrightarrow}0.
\end{equation}
By defining $\mathcal E'$ exactly as $\mathcal E$ but by replacing $\mathcal R$ with $\mathcal R'$, we have thus shown that 
\begin{equation}
    \left|\frac{h(n)^2b_\beta(n)^d}{m(n)n^2}\left\{\mathcal E_{nt}'\right\}+t^\chi\gamma_t'^\beta-\chi\int_0^t\left(\gamma_t'^\beta-\gamma_{t-s}'^\beta\right)s^{\chi-1}\,\d s\right|\overset{a.s.}{\underset{n\rightarrow\infty}\longrightarrow}0.
\end{equation}
Therefore, we deduce that 
\begin{equation*}
    \frac{h(n)b_\beta(n)^d}{m(n)n^2}\{\mathcal E_{nt}\}\distribution -t^\chi\gamma_t^\beta+\chi\int_0^t\left(\gamma_t^\beta-\gamma_{t-s}^\beta \right)s^{\chi-1}\,\d s.
\end{equation*}
Finally, the fact that 
\begin{equation*}
    -t^\chi\gamma_t^\beta+\chi\int_0^t\left(\gamma_t^\beta-\gamma_{t-s}^\beta\right)s^{\chi-1}\,\d s=-\int_0^t(t-s)^\chi\,\d\gamma_s^\beta
\end{equation*}
in the sense of Young is simply (\ref{eq:definition of singular Young integral}) for $g=\gamma^\beta$.

\subsection{Extension to a Functional CLT}

\label{section:FLT energy}

As announced in the Introduction, we now extend Theorem \ref{theorem:main result 2} to its functional analog

\begin{theorem}
Under the assumptions of Theorem \ref{theorem:main result 2}, the following convergences in distribution hold on $\mathcal C(\R_+)$ endowed with the topology of uniform convergence on compacts
    \begin{itemize}
        \item If $d/\beta\ge3/2$ and $\chi>-1/2$,
        \begin{equation}
            \left(\left(m(n)\sqrt{ng(n)}\right)^{-1}(\mathcal E_{nt}-\E[\mathcal E_{nt}])\right)_{t\ge0}\distribution\left(\sigma\int_0^t(t-s)^\chi\,\d W_{s}\right)_{t\ge0},
        \end{equation}
        \item If $1\le d/\beta<3/2$ and $\chi>\beta/d-2$,
        \begin{equation}
            \left(\frac{h(n)^2b_\beta(n)^d}{m(n)n^2}\left(\mathcal E_{nt}-\E[\mathcal E_{nt}]\right)\right)_{t\ge0}\distribution \left(-\int_0^t(t-s)^\chi\,\d\gamma_s^\beta\right)_{t\ge0},
        \end{equation}
        \item If $d/\beta<1$ and $\chi>-1/\beta$,
        \begin{equation}
            \left((m(n)b_\beta(n))^{-1}\mathcal E_{nt}\right)\distribution\left(\int_0^t(t-s)^\chi\,\d L(s)\right)_{t\ge0},
        \end{equation}
    \end{itemize}
\end{theorem}

\begin{proof}
    As previously, we only focus on the case $1\le d/\beta<3/2$, since the other cases are treated identically. The convergence of finite dimensional marginals is obtained exactly as in the proof of the CLT, we leave the details to the reader. To show tightness, we introduce for $t\ge0$, $n\ge1$,
    \begin{equation*}
        \overline{\mathcal E}_t^{(n)}:=\frac{h(n)^2b_\beta(n)^d}{m(n)n^2}\{\mathcal E_{nt}\}.
    \end{equation*}
    We fix $T>0$ and let $0\le s<t\le T$ and take $\eta>0$ sufficiently small. Recalling that
    \begin{equation*}
        \overline{\mathcal E}_t^{(n)}=m_n(t)\overline{\mathcal R}_t^{(n)}-\int_0^t\overline{\mathcal R}_{s,t}^{(n)}\partial_sm_n(s)\,\d s,
    \end{equation*}
    we get 
    \begin{align*}
        \left|\overline{\mathcal E}_t^{(n)}-\overline{\mathcal E}_s^{(n}\right|&\le\left|m_n(t)\overline{\mathcal R}_t^{(n)}-m_n(s)\overline{\mathcal R}_s^{(n)}\right|+\left|\int_s^t\overline{\mathcal R}_{u,t}^{(n)}\partial_um_n(u)\,\d u\right|\\
        &\quad + \left|\int_0^s\left(\overline{\mathcal R}_{u,t}^{(n)}-\overline{\mathcal R}_{u,s}^{(n)}\right)\partial_um_n(u)\,\d u\right|\\
        &\le \left|m_n(t)-m_n(s)\right|\left|\overline{\mathcal R}_t^{(n)}\right|+|m_n(s)|\left|\overline{\mathcal R}_{t-s,t}^{(n)}\right|+\left|\int_s^t\overline{\mathcal R}_{u,t}^{(n)}\partial_um_n(u)\,\d u\right|\\
        &\quad + \left|\int_0^s\left(\overline{\mathcal R}_{t-s,t}^{(n)}-\overline{\mathcal R}_{t-s,t-u}^{(n)}\right)\partial_um_n(u)\,\d u\right|
    \end{align*}
    By recalling (\ref{ineq:bound for m_n'}), Corollary \ref{corollary:local uniform Hölder continuity of scaled range}, the fact that $m_n(s)$ converges to $s^\chi$ as $n\rightarrow\infty$ and the mean value Theorem, we may find some $a.s.$ finite random variable $C(T,\eta)$ depending only on $T$ and $\eta$ such that 
    \begin{align*}
        \left|\overline{\mathcal E}_t^{(n)}-\overline{\mathcal E}_s^{(n}\right|&\le C(T,\eta)\left(|t-s|+|t-s|^{2-d/\beta-\eta}+|t-s|+2|t-s|^{2-d/\beta-\eta}\right)\\
        &\le C(T,\eta)\left(2T^{d/\beta+\eta}+3\right)|t-s|^{2-d/\beta-\eta}=:\tilde C(T,\eta)|t-s|^{2-d/\beta-\eta}.
    \end{align*}
    We let $\mathcal C_{L,T}^\alpha\subseteq \mathcal C([0,T])$ denote the $\alpha$-Hölder functions of $\mathcal C([0,T])$ with $\alpha$-Hölder norm bounded by $L$. Each $\mathcal C_{L,T}^\alpha$ is relatively compact in $\mathcal C([0,T])$ by Ascoli's theorem and by letting $\varepsilon>0$, we have for $L>0$, $n\ge1$,
    \begin{equation*}
        \p\left(\overline{\mathcal E}_{\mid[0,T]}^{(n)}\notin\mathcal C_{L,T}^{2-d/\beta-\eta}\right)=\p\left(\tilde C(T,\eta)>L\right),
    \end{equation*}
    which is made smaller than $\varepsilon$ by letting $L$ become sufficiently large. We have thus shown tightness of the $\overline{\mathcal E}_{\mid[0,T]}^{(n)}$ for any $T>0$, which concludes the proof.
\end{proof}

\appendix

\section{Young Integral}
\label{appendix:Young integral}

The aim of the Young Integral (\cite{Young1936}, \cite{CASTREQUINI2022112064}) is to extend the Stieljtes integral $\int f\,\d g$ (which requires $f$ to be continuous and $g$ to be of bounded variation) to a case where $g$ is supposed to be more irregular, which naturally imposes some stronger regularity on $f$. Precisely, if $f$ and $g$ are respectively $\alpha$-Hölder and $\beta$-Hölder on some compact interval $[a,b]$ where $\alpha+\beta>1$ (which we suppose to be the case in the rest of the section), then the Young integral of $f$ against $\,\d g$ is defined as the limit of the Riemann sums
\begin{equation*}
    \int_a^bf(t)\,\d g(t)=\underset{\Delta(\pi)\rightarrow0}\lim\underset{\pi}\sup\sum_{i=1}^{|\pi|}f(t_i)\left(g(t_i)-g(t_{i-1})\right),
\end{equation*}
where the supremum is taken over all subdivisions $\pi=\{a=t_0<t_1<\dots<t_{|\pi|}=b\}$ of $[a,b]$ and where $\Delta(\pi)$ is the mesh of $\pi$. The mapping $(f,g)\mapsto\int f\,\d g$ is bilinear and the classical integration by parts formula holds :
\begin{equation*}
    \int_a^bf(t)\,\d g(t)=f(a)g(a)-f(b)g(b)-\int_a^bg(t)\,\d f(t).
\end{equation*}
Furthermore, if we define $g_b(t):=g(b-t)$ for all $t\in[a,b]$, then $g_b$ is $\beta$-Hölder and by noticing that for any subdivision $\pi=\{t_i\}$, the set $\tilde\pi:=\{b-t_i\}$ is still a subdivision of $[a,b]$ and that the mapping $\pi\mapsto\tilde\pi$ is bijective, one easily checks by examining the definition that we have the following time inversion formula :
\begin{equation*}
    \int_a^bf(t)\,\d g(t)=-\int_a^bf(b-t)\,\d g_b(t).
\end{equation*}
We now take $f$ of the form $f(s):=(t-s)^\chi$ for some $t>0$ and $\chi>-\beta$. If $\chi\ge0$, then $f$ is Lipschitz on $[0,t]$ and so the Young integral $\int_0^t(t-s)^\chi\,\d g(s)$ is well defined. If $\chi<0$, then $f$ is only Lipschitz on $[0,t-\varepsilon]$ for any $\varepsilon\in(0,t)$, and so the Young integral $\int_0^{t-\varepsilon}(t-s)^\chi\,\d g(s)$ is well defined. Furthermore, by time inversion, bilinearity and the integration by parts formula, an immediate calculation yields
\begin{equation}
\label{eq:manipulation for singular integral with cutoff}
    \int_0^{t-\varepsilon}(t-s)^\chi\,\d g(s)=t^\chi g(t)-\varepsilon^\chi(g(t)-g(t-\varepsilon))-\int_\varepsilon^t(g(t)-g(t-s))\,\d\left(s^\chi\right).
\end{equation}
Firslty, the $\beta$-Hölder regularity of $g$ yields for some $C>0$
\begin{equation}
\label{ineq:bound for normal term in cutoff integral}
|\varepsilon^\chi(g(t)-g(t-\varepsilon))|\le C\varepsilon^{\chi+\beta},
\end{equation}
which tends to $0$ as $\varepsilon\rightarrow0$, since $\beta+\chi>0$. Furthermore, the monotone convergence Theorem yields
\begin{equation}
    \label{eq:integral term of singular integral with cutoff}
    \underset{\varepsilon\rightarrow0}\lim\int_\varepsilon^t(g(t)-g(t-s))\,\d\left( s^\chi\right)=\chi\int_0^t(g(t)-g(t-s))s^{\chi-1}\,\d s,
\end{equation}
where the right-hand side is again well defined thanks to the $\beta$-Hölder regularity of $g$. Combining (\ref{ineq:bound for normal term in cutoff integral}) and (\ref{eq:integral term of singular integral with cutoff}), we see that the left-hand side of (\ref{eq:manipulation for singular integral with cutoff}) converges as $\varepsilon\rightarrow0$ and we define 
\begin{equation}
\label{eq:definition of singular Young integral}
    \int_0^t(t-s)^\chi\,\d g(s):=\underset{\varepsilon\rightarrow0}\lim\int_0^{t-\varepsilon}(t-s)^\chi\,\d g(s)=t^\chi g(t)-\chi\int_0^t(g(t)-g(t-s))s^{\chi-1}\,\d s
\end{equation}

\paragraph{Acknowledgement.} This work was funded by the CNRS project MITI. The author wishes to warmly thank his advisors Vincent Bansaye, Jean-René Chazottes and Sylvain Billiard for the time spent discussing and proofreading this work.

\printbibliography[title={Bibliography}]

\end{document}